\documentclass[12pt]{amsart}
\usepackage{amssymb,latexsym, amsmath, amsthm,pdfsync,gensymb,upgreek}
\usepackage{enumerate}

\usepackage{graphicx}
\usepackage[mathscr]{eucal}
\usepackage[usenames]{color}
\definecolor{Purple}{rgb}{.7,0.08,0.6} 
\usepackage{layout}
\usepackage{setspace}
\usepackage{nicefrac}
\usepackage[font= small]{caption}
\theoremstyle{plain}

\newtheorem{Thm}{Theorem}
\newtheorem{Cor}{Corollary}
\newtheorem{Lem}{Lemma}
\newtheorem{Prop}{Proposition}

\theoremstyle{Definition}
\newtheorem{Def}{Definition}
\newtheorem{Ex}{Example}
\newtheorem{Rmrk}{Remark}
\newtheorem{BP}{Basic Property}

\newcommand{\dee}{\partial}
\newcommand{\deebar}{\overline\partial}

\renewcommand{\)}{]}

\newcommand{\B}{B}
\newcommand{\bndry}{b}
\newcommand{\D}{\mathcal D}
\newcommand{\Bop}{\mathbf B}

\newcommand{\g}{g}
\newcommand{\F}{F}
\newcommand{\lam}{\zeta}
\renewcommand{\H}{H}
\newcommand{\s}{\mathcal S}
\newcommand{\Sz}{C}

\newcommand{\V}{V}
\numberwithin{equation}{section}
\begin{document}
\title[Cauchy-type Integrals]{Cauchy-type integrals in several complex variables\footnote{This is an electronic reprint of the original article published in the Bulletin of the Mathematical Sciences, vol. 3 (2) (2013) 241-285. This reprint differs from the original in pagination and typographical detail.}}
\author[Lanzani and Stein]{Loredana Lanzani$^*$
and Elias M. Stein$^{**}$}
\thanks{$^*$ Supported in part
by the National Science Foundation, award DMS-1001304}
\thanks{$^{**}$ Supported in part by the National Science Foundation, award DMS-0901040, and by KAU of Saudi Arabia} 
\address{
Dept. of Mathematics,       
University of Arkansas 
Fayetteville, AR 72701}
\address{
Dept. of Mathematics\\Princeton University 
\\Princeton, NJ   08544-100 USA }
\smallskip

  \email{loredana.lanzani@gmail.com,\; stein@math.princeton.edu}
\begin{abstract} 
We present the theory of Cauchy-Fantappi\'e integral operators,
with emphasis on the situation when the
domain of integration, $D$, has minimal boundary regularity.
Among these operators we focus on those that  are more closely related to
 the classical Cauchy integral for a planar domain,
whose kernel is a holomorphic function of the parameter $z\in D$.
The goal is to prove 
$L^p$ estimates for these operators and, as a consequence, to obtain $L^p$ estimates for the canonical Cauchy-Szeg\"o and Bergman projection operators (which are not of Cauchy-Fantappi\'e type).
\end{abstract}
\maketitle
\section{Introduction}

The purpose of this survey is to study Cauchy-type integrals in several complex variables and to announce new results concerning these operators.
While this is a broad field with a very wide literature, our exposition will be focused more narrowly
on achieving the following goal: the construction of such operators and the establishment of their
$L^p$ mapping properties under ``minimal'' conditions of smoothness of the boundary of the domain $D$ in question.\\

The operators we study are of three interrelated kinds: Cauchy-Fantappi\' e integrals with holomorphic kernels, Cauchy-Szeg\"o projections and Bergman projections. In the case of one complex variable, what happens is by now well-understood. Here the minimal smoothness that can be achieved is ``near'' $C^1$ (e.g., the case of a Lipschitz domain). However when the complex dimension is greater than 1 the nature of the Cauchy-Fantappi\' e kernels brings in 
considerations of pseudo-convexity (in fact strong pseudo-convexity) and these in turn imply that the limit of smoothness
should be ``near'' $C^2$.\\

We establish $L^p$-regularity for one or more of these operators in the following contexts:

\begin{itemize}
\item When $D$ is strongly pseudo-convex and of class $C^2$\,;
\item When $D$ is strongly $\mathbb C$-linearly convex and of class $C^{1, 1}$
\end{itemize}

with $p$ in the range $1<p<\infty$. See theorems \ref{T:CL-bdd-smooth} - 
\ref{T:CL-bdd-C11} for the precise statements.\\

This survey is organized as follows. In Section \ref{S:planar} we briefly review the situation of one complex variable. Section \ref{S:CI-n} is devoted to a few generalities  about Cauchy-type integrals when $n$, the complex dimension of the ambient space, is greater than 1. The Cauchy-Fantappi\' e forms are taken up in Section \ref{S:CF-n} and the corresponding Cauchy-Fantappi\'e integral operators are set out in Section \ref{S:poten-th}. Here we adapt the standard treatment
in \cite[Chapt. IV]{Ra}, but our aim is to show that these methods apply when the so-called generating form is merely of class $C^1$ or even Lipschitz, as is needed in what follows.

The Cauchy-Fantappi\'e integrals constructed up to that point may however lack the basic requirement of producing holomorphic functions, whatever the given data is. In other words, the kernel of the operator may fail to be holomorphic in the free variable $z\in D$. To achieve 
the desired holomorphicity requires that the domain $D$ be pseudo-convex, and two specific forms of this property, {\em strong pseudo-convexity} and {\em strong $\mathbb C$-linear convexity} are discussed in Section \ref{S:convexity}.

There are several approaches to obtain the required holomorphicity when the domain is sufficiently smooth and strongly pseudo-convex.The initial methods are due to Henkin 
\cite{Hen-1}-\cite{Hen-2} and Ramirez \cite{R}; a later approach is in Kerzman-Stein \cite{KS}, which is the one we adopt here.
It requires to start with a ``locally'' holomorphic kernel, and then to add a correction term obtained by solving a $\deebar$-problem. These matters are discussed in Section \ref{S:local-holom}-\ref{S:global-hol}. One should note that in the case of strongly $\mathbb C$-linearly convex domains, the Cauchy-Leray integral given here requires no correction. So among all the integrals of Cauchy-Fantappi\' e type associated to such domains, the Cauchy-Leray integral is the unique and natural operator
 that most closely resembles the classical Cauchy integral
from one complex variable.

The main $L^p$ estimates
 for the Cauchy-Leray integral and the Szeg\"o and Bergman projections 
 (for $C^2$ boundaries) 
 are the subject of a series of forthcoming papers; in Section 
\ref{S:main} we limit
ourselves to 
 highlighting the main points of interest in the proofs. 
For the last two operators, the $L^p$ results are
 consequences of estimates that hold for the corrected 
Cauchy-Fantappi\'e kernels, denoted $\Sz_\epsilon$ and $\B_\epsilon$, that involve also their respective adjoints. Section \ref{S:last} highlights a further result concerning the Cauchy-Leray
integral, 
also to appear in a separate paper: the corresponding $L^p$ theorem under the 
weaker assumption that the boundary is merely of class $C^{1,1}$.\\

A survey of this kind must by the nature of the subject be far from complete. Among matters not 
covered here are $L^p$ results for the Szeg\"o and Bergman projection and for the Cauchy-Leray integral for other special domains
 (in particular, with more regularity). For these, see e.g. \cite{B}, \cite{B1}, \cite{BaLa}, \cite{BeLi},\cite{BoLo}, \cite{CD}, \cite{EL}, \cite{F}, \cite{H}, \cite{KP}, \cite{MS-1}, \cite{MS-2}, \cite{PS}, \cite{Z}. It is to be noted that several among these works depend in the main on good estimates or explicit formulas for the Szeg\"o or Bergman kernels. In our situation these are unavailable, and instead we have to proceed via the Cauchy-Fantappi\' e framework.\\

A few words about notation:
\begin{itemize}
\item[]  
 Euclidean volume measure for $\mathbb C^n \equiv \mathbb R^{2n}$ ($n\geq 1$) will be denoted $dV$.
\smallskip

\item[]  The notation $\bndry D$ will indicate the boundary of a domain $D\subset \mathbb C^n$ ($n\geq 1$) and, for $D$ sufficiently smooth, 
$d\sigma$ will denote arc-length ($n=1$) or Euclidean surface measure ($n\geq 2$).
\end{itemize}
\bigskip

\noindent {\em Aknowledgment.}\ We wish to express our gratitude to David Barrett, for helpful conversations, and to Annalisa Calini for producing the illustrations in sections \ref{S:correction} and \ref{S:global-hol}.
\bigskip

\section{The Case $n=1$}\label{S:planar}

In the case of one complex dimension the problem of $L^p$ estimates has a long and illustrious history.
Let us review it briefly. (Some details can be found in \cite{D}, \cite{He}, \cite{LS-1}, which contain further citations to the literature.) \\

Suppose $D$ is a bounded domain in $\mathbb{C}$ whose boundary $\bndry D$ is a rectifiable curve. Then the {\em Cauchy integral} is given by
\begin{equation*}
   \mathbf \Sz(f)(z) = \int_{\bndry D}\limits \!\! f(w)\, C(w, z) \ , \ \text{ for } \ z \in D 
\end{equation*}
where $C(w, z)$ is the {\em Cauchy kernel}
\begin{equation*}
C(w, z) = \frac{1}{2\pi i}\,\frac{dw}{w-z}
\end{equation*}
When $D$ is the unit disc, then a classical theorem of M. Riesz says that the mapping $f \mapsto \mathbf \Sz(f) |_{\bndry D}$, defined initially for $f$ that are (say) smooth, is extendable to a bounded operator on $L^p (\bndry D )$, for $1 < p < \infty$. Very much the same result holds when the boundary of $D$ is of class $C^{1 + \epsilon}$, with $\epsilon > 0$, (proved either by approximating to the result when $D$ is the unit disc, or adapting one of the several methods of proof used in the classical case). However in the limiting case when $\epsilon = 0$, these ideas break down and new methods are needed. The theorems proved by Calder\'{o}n, Coifman, McIntosh, Meyers and David (between 1977-1984) showed that the corresponding $L^p$ result held in the following list of increasing generality: the boundary is of class $C^1$; it is Lipschitz (the first derivatives are merely bounded and not necessarily continuous); it is an \textquotedblleft Ahlfors-regular\textquotedblright \ curve.\\

We pass next to the Cauchy-Szeg\"o projection $\mathbf S$, the corresponding orthogonal projection with respect to the Hilbert space structure of $L^2 (\bndry D )$. In fact when $D$ is the unit disc, the two operators $\mathbf \Sz$ and $\mathbf S$ are identical. Let us now restrict our attention to the case when $D$ is simply connected and when its boundary is Lipschitz. Here a key tool is the conformal map $\Phi : \mathbb D \rightarrow D$, where $\mathbb D$ is the unit disc. We consider the induced correspondence $\tau$ given by $\displaystyle \tau (f)(e^{i \theta}) = 
( \Phi^\prime (e^{i \theta }))^{\frac{1}{2}} f(\Phi (e^{i \theta}))$, and the fact that $S = \tau^{-1} S_0 \tau$, where $S_0$ is the Cauchy-Szeg\"o projection for the disc $\mathbb D$. Using ideas of Calder\'{o}n, Kenig, Pommerenke and others, one can show that $| \Phi^\prime |^r$ belongs to the Muckenhaupt class $A_p$, with $r=1 - p/2$, from which one gets the following. As a consequence, if we suppose that $\bndry D$ has a Lipschitz bound $M$, then $\mathbf S$ is bounded on $L^p$, whenever
\begin{itemize}
\item $1 < p < \infty$, if $\bndry D$ is in fact of class $C^1$.
   \item $p^\prime_M < p < p_M$. Here $p_M$ depends on $M$, but $p_M > 4$, and $p^\prime_M$ is the exponent dual to $p_M$.
\end{itemize}
\indent There is an alternative approach to the second result that relates the Cauchy-Szeg\"o projection $\mathbf S$ to the Cauchy integral $\mathbf \Sz$. It is based on the following identity, used in \cite{KS}
\begin{equation}\label{Eq:1}
   \mathbf S(I-\mathbb A) = \mathbf \Sz \ , \ \ \text{ where} \ \mathbb A = \mathbf \Sz^\ast - \mathbf \Sz .
\end{equation} 
\indent There are somewhat analogous results for the Bergman projection in the case of one complex dimension. We shall not discuss this further, but refer the reader to the papers cited above.

\bigskip

\section{Cauchy integral in $\mathbb C^n$, $n>1$; some generalities}\label{S:CI-n}
We shall see that a very different situation occurs when trying to extend the results of Section \ref{S:planar} to higher dimensions. Here are some new issues that arise when $n > 1$:
\begin{enumerate}
   \item[\tt{i.}] There is no \textquotedblleft universal\textquotedblright \ holomorphic Cauchy kernel associated to a domain $D$.
\item[\tt{ii.}] Pseudo-convexity of $D$, must, in one form or another, play a role.
\item[\tt{iii.}] Since this condition involves (implicitly) two derivatives, the \textquotedblleft best\textquotedblright \ results are to be expected \textquotedblleft near\textquotedblright \ $C^2$, (as opposed to near $C^1$ when $n=1$).
\end{enumerate}
\indent In view of the non-uniqueness of the Cauchy integral (and its problematic existence), it might be worthwhile to set down the minimum conditions that would be required of candidates for the Cauchy integral. We would want such an operator $\mathbf \Sz$ given in the form
\begin{equation*}\label{Eq:2}
   \mathbf \Sz(f)(z) = \int_{\bndry D}\limits\!\! f(w)\,C(w, z), \quad z \in D ,
\end{equation*} 
to satisfy the following conditions:\\
\begin{enumerate}
   \item[(a)] The kernel $C(w, z)$ should be given by a \textquotedblleft natural\textquotedblright \ or explicit formula (at least up to first approximation) that involves $D$. 
   \item[]
\item[(b)] The mapping $f \mapsto \mathbf \Sz(f)$ should reproduce holomorphic functions. In particular if $f$ is continuous in $\overline{D}$ and holomorphic in $D$ then $\mathbf \Sz(f)(z) = f(z)$, for $z \in D$.
\item[]
\item[(c)] $\mathbf \Sz(f)(z)$ should be holomorphic in $z \in D$, for any given $f$ that is continuous on $\bndry D$.
\item[]
\end{enumerate}
\indent Now there is a formalism (the Cauchy-Fantappi\'{e} formalism of Fantappi\'{e} (1943), Leray, and Koppleman (1967)), which provides Cauchy integrals satisfying the requirements (a) and (b) in a general setting. Condition (c) however, is more problematic, even when the domain is smooth. Constructing such Cauchy integrals has been carried out only in particular situations, (see below).

\bigskip

\section{Cauchy-Fantappi\'e Theory in higher dimension}\label{S:CF-n}

The Cauchy-Fantappi\'e formalism that realizes the kernel $C(w, z)$ revolves around the notion of {\em generating form}:
these are a class of differential forms of type $(1, 0)$ in the variable of integration whose coefficients may depend on two sets of variables ($w$ and $z$),
and we will accordingly write
\begin{equation*}
\eta (w, z) \, =\, \sum\limits_{j=1}^n\eta_j(w, z)\,dw_j\quad  \mbox{with}\ (w, z)\
\in\  U\times V
\end{equation*}
to designate such a form. The precise definition is given below, where the notation
\begin{equation*}\label{D:pairing}
\langle\eta (w, z), \xi\rangle = \sum\limits_{j=1}^{n} \eta_j(w, z)\,\xi_j.
\end{equation*}
is used to indicate the action of $\eta$ on the vector $\xi\in \mathbb C^n$.
\begin{Def}\label{D:gen-form}
The form
$\eta (w, z)$
is
{\em generating at $z$ relative to $V$}
 if there is an open set 
$$U_z\subseteq\mathbb C^n\setminus\{z\}$$
 such that 
\begin{equation}\label{D:generating-2}
\bndry V\subset U_z
\end{equation}
and, furthermore
\begin{equation}\label{D:generating-1}
\langle\eta (w, z), w-z\rangle = \sum\limits_{j=1}^{n} \eta_j(w, z)\,(w_j-z_j)\equiv 1\quad \text{for any}\ \ w\in 
 U_z.
\end{equation}
 We say that $\eta$ is a 
{\em generating form for $V$} (alternatively, that {\em $V$ supports a generating form $\eta$})
if for any $z\in V$ we have that $\eta$ is generating at $z$ relative to $V$.
\end{Def}
\begin{Ex}\label{Ex:BM-n} Set
\begin{equation*}
\beta (w, z) = |w-z|^2
\end{equation*}
We define the \emph{Bochner-Martinelli generating form}
 to be
\begin{equation}\label{D:BM}
\eta (w, z) \, = \,
\frac{\dee_w\beta}{\beta}(w, z)
\, =\, 
\sum\limits_{j=1}^n\,\frac{\overline{w}_j- \overline{z}_j}{|w-z|^2}\,dw_j\ 
\end{equation}
 It is clear that $\eta$ satisfies conditions \eqref{D:generating-2} and \eqref{D:generating-1}
for any domain $V$ and for any $z\in V$, with 
$U_z:=\mathbb C^n\setminus\{z\}$. 
\end{Ex}
The Bochner-Martinelli generating form 
 has several remarkable features. First, it is ``universal'' in the sense
that it is given by a formula \eqref{D:BM} that does not depend  on the choice of domain $V$; 
 secondly, in complex dimension $n=1$
 it  agrees (up to a scalar multiple) with the classical Cauchy kernel
$$
\frac{1}{2\pi i}\frac{dw}{w-z}, \quad w\in U_z:= \mathbb C\setminus \{z\}
$$
 and in particular its coefficient $(w-z)^{-1}$
  is holomorphic as a function of  $z\in V$ for any fixed $w\in\bndry V$.
 On the other hand, it is clear from \eqref{D:BM} that
 for
  $n\ge 2$ the coefficients of this
   form are nowhere holomorphic:
   this failure at holomorphicity 
    is an instance of a
 crucial,
dimension-induced 
phenomenon 
 which was alluded to in conditions {\tt{ii.}} and (c) in Section \ref{S:CI-n} and will be further discussed in Example \ref{Ex:BM-CF} below and in Section \ref{S:convexity}.
\subsection{Cauchy-Fantappi\'e forms.}
Suppose now that for each fixed $z$
$\eta (w, z)$ is a form of type $(1,0)$ in $w$ with coefficients of class $C^1$ in each variable.
(We are not assuming that $\eta$ is a generating form).
 Set

\begin{equation}\label{D:CF-0}
\Omega_0(\eta)(w, z) = \frac{1}{(2\pi i)^n}\,\eta\wedge(\deebar_w\eta)^{n-1}(w, z)
\end{equation}
where $(\deebar_w\eta)^{n-1}$ stands for the wedge product: $\deebar_w\eta\,\wedge\cdots\wedge\,\deebar_w\eta$ performed $(n-1)$-times.
We call $\Omega_0 (\eta)$
the {\it Cauchy-Fantappi\`e form
for $\eta$}.
Note that $\Omega_0(\eta)(w, z)$ is
 of type $(n, n-1)$ in the variable $w\in U$ while in the variable $z\in V$ it is just a function. 
\medskip

\begin{Ex}\label{Ex:BM-CF} The {\em Cauchy-Fantappi\'e form for the Bochner-Martinelli generating form} or, for short,  {\em Bochner-Martinelli CF form} is 
\begin{equation*}\label{E:BM-CF}
\Omega_0\!\left(\frac{\dee_w\beta}{\beta}\right)\!(w, z) =
\frac{(n-1)!}{(2\pi i\,|w-z|^2)^n}\sum\limits_{j=1}^n
(\overline w_j - \overline z_j)\, dw_j\wedge\left(\bigwedge\limits_{\nu\neq j}d\overline{w}_\nu\wedge dw_\nu\right).
\end{equation*}
Now the {\em Bochner-Martinelli integral} is the operator
\begin{equation*}
\mathbf{\Sz_{BM}}f(z)=\int\limits_{w\in\bndry D}\!\!\!\!\! f(w)\,\Sz_{BM}(w, z),\quad z\in D,\ \ f\in C(\bndry D)
\end{equation*}
where the kernel $\Sz_{BM}(w, z)$ is the Bochner- Martinelli CF form restricted to the boundary; more precisely
\begin{equation*}\label{D:BM-CF-form}
\Sz_{BM}(w, z) = j^*\Omega_0\!\left(\frac{\dee_w\beta}{\beta}\right)\!(w, z),\quad w\in\bndry D,\ \ z\in D
\end{equation*}
where $j^*$ denotes the pullback of forms under the inclusion
$$j: \bndry D\hookrightarrow\mathbb C^n$$
\end{Ex}
It is clear that such operator is ``natural'' in the sense discussed in condition (a) in Section \ref{S:CI-n}, and we will see in Section \ref{S:poten-th} that this operator also satisfies condition
(b), see proposition \ref{P:repr-BM-CF-1} in that section. On the other hand, the kernel $\Sz_{BM}(w, z)$ is nowhere holomorphic in $z$: as a result, when $n>1$ the Bochner-Martinelli integral does not satisfy condition (c).\\

We will now review the properties of Cauchy-Fantappi\`e forms that are most relevant to us.
\begin{BP}\label{BP-1} 
 For any function $g\in C^1(U)$ 
we have
\begin{equation*}\label{E:BP-1}
\Omega_0\!\left(g(w)\,\eta(w, z)\right) = g^n(w)\,\Omega_0(\eta)(w, z)\quad \mbox{for any}\ w\in U.
\end{equation*}
\end{BP}
\begin{proof} The proof is a computation: by the definition \eqref{D:CF-0}, we have
$$
\Omega_0\!\left(g\,\eta\right) = g\,\eta\wedge \left(\deebar(g\,\eta)\right)^{n-1}
$$
On the other hand, computing $\left(\deebar(g\, \eta)\right)^{n-1}$ produces two kinds
of terms:
\begin{itemize}
\item[(a.)] Terms that contain $\deebar g\wedge\eta$: but these do not contribute to $\Omega_0(g\,\eta)$
because $g\,\eta\wedge\deebar g\wedge\eta =0$ (which follows from $\eta\wedge\eta=0$ since $\eta$ has degree $1$);
\item[(b.)] The term $g^{n-1}\,\deebar\,\eta$, which gives the desired conclusion.
\end{itemize}
\end{proof}

Suppose, further, that $\eta(w, z)$ is generating at $z$ relative to $V$. Then the following
two properties also hold.
\begin{BP}\label{BP-2} We have that
\begin{equation}\label{E:BP-2}
(\deebar_w\eta)^{n}(w, z)
= 0\quad \text{for any}\  w\in U_z.
\end{equation}
\end{BP}
Note that if the coefficients of $\eta(\cdot, z)$ are in $C^2(U_z)$, then as a consequence of the fact that $\deebar\circ\deebar =0$, we have that $(\deebar_w\eta)^{n}(w, z) = d_w\Omega_0(\eta)(w, z)$ and 
\eqref{E:BP-2}
can be formulated equivalently as
\begin{equation*}\label{E:BP-2aux}
d_w\Omega_0(\eta)(w, z)\ =\ 0,\quad w\in U_z.
\end{equation*}
\begin{proof}
We prove \eqref{E:BP-2} in the case: $n=2$ and leave the proof for general
$n$ as an exercise for the reader. 
Thus, writing $$\eta = \eta_1dw_1+\eta_2dw_2$$ we obtain
\begin{equation}\label{E:deebar-eta-square}
(\deebar_w\eta)^2= -2\,\deebar_w\eta_1\wedge\deebar_w\eta_2\wedge dw_1\wedge dw_2
\end{equation}
Now
\begin{equation*}\label{E-gen-2-vble}
\eta_1(w, z)(w_1-z_1) + \eta_2(w,z)(w_2-z_2) = 1\quad \mbox{for any}\quad w\in U_z
\end{equation*}
because $\eta$ is generating at $z$, and applying $\deebar_w$ to each side of this identity we obtain
\begin{equation}\label{E:d-gen-2-vble}
(w_1-z_1)\,\deebar_w\eta_1(w, z) + (w_2-z_2)\,\deebar_w\eta_2(w, z) = 0\quad \mbox{for any}\quad w\in U_z
\end{equation}
Recall that $U_z\subset\mathbb C^2\setminus\{z\}$, see definition \ref{D:gen-form}, and so
\begin{equation*}
U_z\cap U = U_z^1 \cup U_z^2
\end{equation*}
where
\begin{equation}\label{E:set-1}
U_z^1 =\{w =(w_1, w_2)\in U_z\cap U,\ w_1\neq z_1\}
\end{equation}
\begin{equation}\label{E:set-2}
U_z^2 =\{w=(w_1, w_2)\in U_z\cap U,\ w_2\neq z_2\}
\end{equation}
But for any two sets $A$ and $B$ one has
$
A\cup B = (A\setminus\! B)\, \dot\cup\, (B\setminus\! A)\,\dot\cup\, (A\cap B)
$
where $\dot\cup$ denotes disjoint union. Now, if $w\in U^1_z\setminus\! U^2_z$ then
\eqref{E:d-gen-2-vble} reads
$$
(w_1-z_1)\,\deebar_w\eta_1(w, z) =0,\quad w_1\neq z_1
$$
but this implies 
$$
\deebar_w\eta_1(w, z) =0\quad \mbox{for any} \quad w\in U^1_z\setminus\! U^2_z.
$$
 One similarly obtains
$$\deebar_w\eta_2(w, z) =0\quad \mbox{ for any}\quad  w\in U^2_z\setminus\! U^1_z). $$
We are left
to consider the case when $w\in U^1_z\cap U^2_z$; note that since
$$
(w_1-z_1)(w_2-z_2)\neq 0\quad \mbox{for any}\quad w\in U^1_z\cap U^2_z
$$
showing that $(\deebar_w\eta)^2(w, z) =0$ for any $w\in U^1_z\cap U^2_z$ is now equivalent 
to showing that
$$
(w_1-z_1)(w_2-z_2)(\deebar_w\eta)^2(w, z) =0 \quad \mbox{for any}\quad w\in U^1_z\cap U^2_z
$$
To this end, combining \eqref{E:deebar-eta-square} with \eqref{E:d-gen-2-vble} we find
$$
(w_1-z_1)(w_2-z_2)(\deebar_w\eta)^2(w, z)=
$$
$$ = 2 (w_1-z_1)^2\,
\deebar_w\eta_1(w, z)\,\wedge \deebar_w\eta_1(w, z)\,\wedge dw_1\wedge dw_2
$$
and indeed $$\deebar_w \eta_1\,\wedge\deebar_w\eta_1 =0$$ because
$\deebar_w\eta_1$ is a form of degree 1.
\end{proof}

\bigskip

 \noindent Let $\eta (w, z)$ be a form of type $(1, 0)$ in the variable $w$  
 (not necessarily generating for $V$) and with coefficients in $C^1(U\times V)$; set

\begin{equation}\label{D:CF-1}
\Omega_1(\eta)(w, z) = \frac{(n-1)}{(2\pi i)^n}\bigg(\eta\wedge(\deebar_w\eta)^{n-2}\wedge\deebar_z\eta\bigg) (w, z)
\end{equation}
\smallskip

Note that $\Omega_1(\eta)(w, z)$ is
of type $(n, n-2)$ in the variable $w$ 
 and of type $(0, 1)$ in the variable $z$. 
We call $\Omega_1 (\eta)$
the {\it Cauchy-Fantappie' form of order 1 for $\eta$}, 
 and the previous
one, $\Omega_0(\eta)$, will now be called {\it Cauchy-Fantappie' form of order 0}.
\bigskip

In the previous properties $z$ was fixed; here it is allowed to vary.
\begin{BP}\label{BP-4} 
We have (again for $\eta$ generating at $z$)
\begin{equation}\label{E:BP-4}
(2\pi i)^{\!n}\,\deebar_z\Omega_0(\eta)(w, z) = -(\deebar_w\eta)^{n-1}\wedge \deebar_z\eta +
\eta\wedge(\deebar_w\eta)^{n-2}\wedge\deebar_z\deebar_w\eta\,,
\end{equation}
For any $w\in U_z\cap U$, where $U_z$ is as in \eqref{D:generating-1}. Note that if the coefficients are in fact of class $C^2$ in each variable, then 
\eqref{E:BP-4} has the equivalent formulation
\begin{equation}\label{E:BP-4-aux}
\deebar_z\Omega_0(\eta)(w, z) = - d_w\Omega_1(\eta)(w, z).
\end{equation}
 \end{BP}
 \begin{proof}
As before, we specialize to the case: $n=2$ and leave the proof of the general case as an exercise for the reader. For $n=2$ identity \eqref{E:BP-4} reads
\begin{equation}\label{E:BP-4-auxx}
\deebar_z\big(\eta\wedge \deebar_w\eta\big) = 
- \deebar_w\eta\wedge \deebar_z\eta + 
\eta\wedge\deebar_z\deebar_w\eta\,
\end{equation}
By the Leibniz rule we have
\begin{equation*}\label{E:2}
\deebar_z\big(\eta\wedge \deebar_w\eta\big) =
 \deebar_z\eta\wedge \deebar_w\eta + \eta\wedge \deebar_z\deebar_w\eta
\end{equation*}
and so it is clear  that 
\eqref{E:BP-4-auxx}
will follow if we can show that
\begin{equation*}\label{E:BP-4-auxxx}
 \deebar_w\eta\wedge \deebar_z\eta =0 ,\quad \mbox{for any}\quad w\in U_z
\end{equation*}
for any generating form $\eta$ with coefficients of class $C^1$. Proceeding as in the proof
of basic property \ref{BP-2}, we decompose 
$$U_z\cap U= U^1_z\cup U^2_z$$ where $U^1_z$ and 
$U^2_z$ are as in \eqref{E:set-1} and \eqref{E:set-2}, respectively. Again, we have
\begin{equation*}\label{E-gen-2-vble-again}
\eta_1(w, z)(w_1-z_1) + \eta_2(w,z)(w_2-z_2) = 1\quad \mbox{for any}\quad w\in U_z
\end{equation*}
because $\eta$ is generating, and applying $\deebar_w$
 to
each side of this identity we find
\begin{equation}\label{E:temp}
0=\left\{\!\!\!\begin{array}{rcl}
\big(\deebar_w\eta_1\big)\!\cdot\! (w_1-z_1)+ 
\big(\deebar_w\eta_2\big)\!\cdot\! (w_2-z_2),&\mbox{if}& w\in U^1_z\cap U^2_z\\
\\
\big(\deebar_w\eta_1\big)\!\cdot \!(w_1-z_1),&\mbox{if}& w\in U^1_z\setminus U^2_z\\
\\
\big(\deebar_w\eta_2\big)\!\cdot\! (w_2-z_2),&\mbox{if}& w\in U^2_z\setminus U^1_z\\
\end{array}
\right. 
\end{equation}
Similarly, applying $\deebar_z$, we have
\begin{equation}\label{E:temp'}
0=\left\{\!\!\!\begin{array}{rcl}
\big(\deebar_z\eta_1\big)\!\cdot\! (w_1-z_1)+
 \big(\deebar_z\eta_2\big)\!\cdot\! (w_2-z_2),&\mbox{if}& w\in U^1_z\cap U^2_z\\
\\
\!\big(\deebar_z\eta_1\big)\!\cdot\! (w_1-z_1),&\mbox{if}& w\in U^1_z\setminus U^2_z\\
\\
\big(\deebar_z\eta_2\big)\!\cdot\! (w_2-z_2),&\mbox{if}& w\in U^2_z\setminus U^1_z\\
\end{array}
\right. 
\end{equation}
Now
\begin{equation}\label{E:temp-1}
\deebar_w\eta\wedge \deebar_z\eta = 
\big(
\deebar_w\eta_1\wedge\deebar_z \eta_2 -
\deebar_w\eta_2\wedge\deebar_z \eta_1
\big)\wedge dw_1\wedge dw_2
\end{equation}
Note that if $w\in U^1_z\setminus U^2_z$ then $w_1\neq z_1$, and so showing
that 
$$
 \deebar_w\eta\wedge \deebar_z\eta =0\quad \mbox{for}\quad w\in U^1_z\setminus U^2_z
$$
is equivalent to showing that
$$
 (\deebar_w\eta\wedge \deebar_z\eta)\cdot(w_1-z_1) =0
$$
that is (using \eqref{E:temp-1})
$$
\big(\deebar_w\eta_1\cdot(w_1-z_1)\wedge\deebar_z \eta_2 -
\deebar_w\eta_2\wedge\deebar_z \eta_1\cdot(w_1-z_1)
\big)\wedge dw_1\wedge dw_2 =0
$$
but this is indeed true by the generating form hypothesis on $\eta$ as expressed in \eqref{E:temp} and \eqref{E:temp'}. This shows that
the desired conclusion is true when $w\in U^1_z\setminus U^2_z$; the case:
$w\in U^2_z\setminus U^1_z$ is dealt with in a similar fashion. Finally, if 
$w\in U^1_z\cap U^2_z$, then $(w_1-z_1)(w_2-z_2)\neq 0$ and
 $$(\deebar_w\eta\wedge \deebar_z\eta)\!\cdot\!(w_1-z_1)(w_2-z_2) =
 $$
 \begin{eqnarray*}
 =\bigg(
 \big(\deebar_w\eta_1\big)\!\cdot\!(w_1-z_1)\wedge
  \big(\deebar_z \eta_2\big)\!\cdot\!(w_2-z_2)+ \\
- \big(\deebar_w\eta_2\big)\!\cdot\!(w_2-z_2)\wedge
 \big(\deebar_z \eta_1\big)\!\cdot\!(w_1-z_1)
 \bigg)\wedge dw_1\wedge dw_2
 \end{eqnarray*}
but the two terms in the righthand side of this identity cancel out on account of \eqref{E:temp}
and  \eqref{E:temp'}.
\end{proof}

\bigskip

\section{reproducing formulas:  some general facts}
\label{S:poten-th}
In this section we highlight the
theory of reproducing formulas for holomorphic functions
 by means of integral operators that arise from the Cauchy-Fantappi\' e formalism. 
 One of our goals here is to show that the usual reproducing properties of such operators extend to the situation where the data and the generating form have lower regularity.
 We begin with a rather specific example: 
 the reproducing formula for 
 the Bochner-Martinelli integral, see 
 proposition \ref{P:repr-BM-CF-1}. The proof of this result is a consequence of a
 recasting of the classical mean value property 
for harmonic functions in terms of an identity \eqref{E:BP-5} that links the Bochner-Martinelli CF form on a sphere with the sphere's Euclidean surface measure.

Because the Bochner-Martinelli integral of a continuous function is, in general, not holomorphic
in $z$, in fact we
 need 
 a more general version of
proposition \ref{P:repr-BM-CF-1} that applies to integral operators whose kernel is allowed to be
 any Cauchy-Fantappi\'e form: this is done in proposition
  \ref{P:repr-hol-bndry}. 
  
  While the operators defined so far are given by surface integrals over the boundary of the ambient domain,  following an idea of Ligocka \cite{L} another family of integral operators can be defined (essentially by differentiating the kernels of the operators in the statement of proposition \ref{P:repr-hol-bndry})  which are realized as ``solid'' integrals over the ambient domain, and we show in
  proposition \ref{P:repr-hol-interior} that such operators, too, have the reproducing property.
 
\begin{Lem}\label{L:BP-3-5} Let $z\in\mathbb C^n$ be given and consider a ball centered at such $z$,\ 
$\mathbb B_r(z) =\{w\in\mathbb C^n,\ |w-z|<r\}$. 

Then, at the center $z$ and
for any $w\in \bndry \mathbb B_r(z)$ we have that the Bochner-Martinelli CF form 
for the ball $\mathbb B_r(z)$
has the following representation
\begin{equation}\label{E:BP-5}
\Sz_{BM}(w, z)=
\frac{d\sigma (w)}{\sigma(\bndry \mathbb B_r(z))}
\end{equation}
where $d\sigma (w)$ is the element of Euclidean surface measure for $\bndry\mathbb B_r(z)$, and
\begin{equation*}
\sigma(\bndry \mathbb B_r(z))=\frac{\,2\pi^nr^{2n-1}}{(n-1)!}
\end{equation*}
denotes surface measure of the sphere $\bndry \mathbb B_r(z)$.
\end{Lem}
\begin{proof} We claim that the desired conclusion is a consequence of the following identity
\begin{equation}\label{E:BP-3}
\Omega_0(\dee_w\beta)(w, z)=  \frac{(n-1)!}{2\pi^n}\,*\dee_w\beta(w, z)
\end{equation}
where, as usual, we have set  $\beta(w, z) = |w-z|^2$, and $*$ denotes the Hodge-star operator mapping forms of type $(p, q)$ to forms of type $(n-q, n-p)$. Let us first prove \eqref{E:BP-5} assuming the truth of
\eqref{E:BP-3}. To this end, we first note that
from \eqref{E:BP-3}
and basic property \ref{BP-1} we have 
\begin{equation*}\label{E:BP-3-aux}
\Omega_0\!\left(\frac{\dee_w\beta}{\beta}\right)\!(w, z) = 
\frac{(n-1)!}{2\pi^n\beta^n}*\dee_w\beta(w, z),\quad w\in \mathbb C^n\, .
\end{equation*}
But 
$ \dee_w\beta(w, z) = \dee\rho (w),$  $w\in\mathbb C^n$
with $\rho(w) := \beta(w, z) -r^2$, a defining function for $\mathbb B_r(z)$.
Now recall 
that $\Sz_{BM}(w, z) = j^*\Omega_0(\dee_w\beta/\beta)$ where
$j$ is the inclusion: $\bndry \mathbb B_r(z)\hookrightarrow\mathbb C^n$, see Example \ref{Ex:BM-CF}, so that
 $j^*\beta^n=r^{2n}$. Combining these facts we conclude that, for $\rho$ as above
\begin{equation*}\label{E:BP-3-aux-1}
\Sz_{BM}(w, z)=
\frac{(n-1)!}{2\pi^n r^{2n}}\,j^*\!(*\dee\rho)(w),\quad w\in \bndry \mathbb B_r(z)
\end{equation*}
and since
$\|d\rho (w)\| = 2r$ whenever $w\in\bndry \mathbb B_r(z)$,
we obtain
\begin{equation*}\label{E:BP-3-aux-2}
\Sz_{BM}(w, z)=
\frac{(n-1)!}{2\pi^n r^{2n-1}}\,\frac{2j^*(*\dee\rho)}{\|d\rho\|}(w),\quad w\in \bndry \mathbb B_r(z);
\end{equation*}
but
\begin{equation}\label{E:surface-meas}
d\sigma(w) = \frac{2j^*\!(*\dee\rho)}{\|d\rho\|}(w),\quad w\in \bndry \mathbb B_r(z)
\end{equation}
see \cite[corollary III.3.5]{Ra}, and this gives \eqref{E:BP-5}.

We are left to prove \eqref{E:BP-3}: to this end, we assume $n=2$ and leave the case of 
arbitrary complex dimension as an exercise to for the reader.
Since
$$
*dw_j =\frac{1}{2\,i^2}dw_j\wedge d\overline{w}_{j'}\wedge dw_{j'},\quad \mbox{where}\ \ j'=\{1, 2\}
\setminus\{j\}
$$
and
$$
\dee_w\beta = (\overline{w}_1-\overline{z}_1)dw_1+ (\overline{w}_2-\overline{z}_2)dw_2
$$
then
$$
*\dee_w\beta = \frac{1}{2i^2}(\overline{w}_1-\overline{z}_1) dw_1\wedge d\overline{w}_2\wedge dw_2 +
(\overline{w}_2-\overline{z}_2) dw_2\wedge d\overline{w}_1\wedge dw_1
$$
On the other hand
$$
\deebar_w\dee_w\beta = d\overline{w}_1\wedge dw_1 + d\overline{w}_2\wedge dw_2
$$
and so
$$
\Omega_0(\dee_w\beta) = \frac{1}{(2\pi i)^2}\,\dee_w\beta\wedge \deebar_w\dee_w\beta =
$$
$$
= \frac{1}{(2\pi i)^2}\bigg((\overline{w}_1-\overline{z}_1)dw_1\wedge d\overline{w}_2\wedge dw_2 +
(\overline{w}_2-\overline{z}_2) dw_2\wedge d\overline{w}_1\wedge dw_1\bigg)=
$$
$$
=\frac{1}{2\pi^2}*\dee_w\beta.
$$
This shows \eqref{E:BP-3} and concludes the proof of the lemma.

(We remark in passing that identity \eqref{E:BP-5}, while valid for the Bochner-Martinelli generating form, is not true for general $\eta$.)
\end{proof}

\begin{Def}\label{D:domain-C-k}
 Given an integer $1\leq k\leq \infty$ and a bounded domain 
$D\subset \mathbb C^n$, we say that {\em $D$ is of class $C^k$} (alternatively,
 \emph{$D$ is $C^k$-smooth})if there is an open neighborhood 
$U$ of the boundary of $D$, and a real-valued function $\rho \in C^k(U)$ such that
\begin{equation*}
U\cap D =\{w\in U\ |\ \rho(w)<0\}
\end{equation*}
and
\begin{equation*}
\nabla\rho (w)\neq 0\quad \mbox{for any}\ w\in U.
\end{equation*}
Any such function is called a {\em defining function for $D$}.
\end{Def}
From this definition it follows that
\begin{equation*}
\bndry D =\{w\in U\ |\  \rho(w)=0\}\quad \mbox{and}
\quad U\setminus \overline{D}=\{w\in U\ |\ \rho(w)>0\}.
\end{equation*}

\begin{Prop}\label{P:repr-BM-CF-1}
For any bounded domain $\V\subset\mathbb C^n$ with boundary of class $C^1$
and for any $f\in\vartheta (\V)\cap C(\overline{\V})$, we have
\begin{equation*}\label{E:repr-C-1-any}
f(z) = \mathbf{\Sz}_{BM}f(z),
\quad z\in \V\, .
\end{equation*}
\end{Prop}
\begin{proof}
Given $z\in V$, let $r>0$ be such that
$$
\overline{\mathbb B_r(z)}\subset V.
$$
Note that the mean value property for harmonic functions:
\begin{equation*}\label{E:MVP}
f(z)\ =\ \frac{1}{\sigma(\bndry\mathbb B_r(z))}\!\!\int\limits_{\bndry\mathbb B_r(z)}\!\!\!\!\! f(w)\,d\sigma(w),\quad f\in\mbox{Harm}(\mathbb B_r(z))\cap C(\overline{\mathbb B_r(z)})
\end{equation*}
and identity \eqref{E:BP-5} give
\begin{equation}\label{E:repr-C-1-ball}
f(z) = 
\int\limits_{w\in \bndry\mathbb B_r (z)}\!\!\!\!\!\!\!\!
f(w)\,\Sz_{BM}(w, z)
\end{equation}
To prove the conclusion, we apply
Stokes' theorem on the set
$$\V_r (z) := \V\setminus \overline{\mathbb B_r (z)}$$
and we obtain
$$
\int\limits_{w\in \V_r(z)}\!\!\!\!\!\!\!\!\!\ d_w\!\left(\!f(w)\,\Omega_0\!\left(\frac{\dee_w\beta}{\beta}(w, z)\!\right)\!\right) =\!\!\!\!
\int\limits_{w\in b\V_r (z)}\!\!\!\!\!\!\!\!f(w)\,\Sz_{BM}(w, z)
$$
But by basic property \ref{BP-2}, and since $f$ is holomorphic, we have
$$
d_w\!\left(\!f(w)\,\Omega_0\!\left(\frac{\dee_w\beta}{\beta}(w, z)\!\right)\!\right) =
f(w)\,\deebar_w\Omega_0\!\left(\frac{\dee_w\beta}{\beta}(w, z)\!\right) =0
$$
and so the previous identity becomes
$$
\int\limits_{w\in b\V}\!\!\!\!\!f(w)\,\Sz_{BM}(w, z)
\ =\!\! \int\limits_{w\in b\mathbb B_r (z)}\!\!\!\!\!\!\!\!f(w)\,\Sz_{BM}(w, z)
$$
but the lefthand side is
$\mathbf{\Sz}_{BM}f(z)$, while  \eqref{E:repr-C-1-ball} says that the righthand side equals $f(z)$.
\end{proof}
\begin{Prop}\label{P:repr-hol-bndry} 
Let $D\subset\mathbb C^n$ be a bounded domain of class $C^1$ and let $z\in D$ be given.
Suppose that $\eta(\cdot, z)$
is a generating form at $z$ relative to $D$.
Suppose, 
furthermore, that 
 the coefficients of $\eta (\cdot, z)$ are in $C^1(U_z)$, where $U_z$ is as in definition \ref{D:gen-form}.
 Then,
 we have
\begin{equation}\label{E:repr-hol-bndry}
f(z) =\!\!\!\!
\int\limits_{w\in \bndry D}\!\!\!\!\!
f(w)\, j^*\Omega_0\!\left(\eta\right)\!(w, z)\quad \mbox{for any}  \ f\in \vartheta (D)\cap C(\overline{D}).
\end{equation}
\end{Prop}
\begin{proof} 
Consider
   a smooth open neighborhood
 of $\bndry D$, which we denote $U_z(\bndry D)$, such that
 \begin{equation}\label{E:fund-hyp-2}
 U_z(\bndry D) \subset U_z
   \end{equation}
  where $U_z$ is as in  \eqref{D:generating-2} and \eqref{D:generating-1}.
Now fix two neighborhoods $U'$ and $U''$  of the boundary of $D$ such that
$$
U''\Subset U'\subset U_z(\bndry D)
$$
and let $\chi_0 (w, z)$ be a smooth cutoff function such that
\begin{equation}\label{E:def-chi}
\chi_0(w, z) =
\left\{
\begin{array}{rcl}
1& \mbox{if}&w\in U''\\
\\
0&\mbox{if}& w\in\mathbb C^n\setminus\overline{ U'}
\end{array}
\right.
\end{equation}
Define
\begin{equation*}
\eta
\degree
 (w, z) = \chi_0(w, z)\eta (w, z) + (1-\chi_0(w, z))\frac{\dee_w\beta}{\beta}(w, z)
\end{equation*}
and 
$$
{D}\degree= D\cap U_z(\bndry D).
$$
Then 
$\eta\degree$ is generating at $z$ relative to $D\degree$
 (and the open set $U_z$ of definition
\ref{D:gen-form}
is the same for $\eta$ and for $\eta\degree$); furthermore, it
follows from \eqref{E:fund-hyp-2} that
$$
\overline{{D}\degree} \subset U_z\, .
$$
Now let $\{\eta_\ell\degree\}_{\ell\in\mathbb N}$
be a sequence of $(1, 0)$-forms with coefficients in $C^2(\overline {D\degree})$
with the property that
\begin{equation*}\label{E:conv-C-1}
\| \eta_\ell\degree - \eta\degree(\cdot, z)\|_{C^1(\overline{D\degree})}\to 0\quad \mbox{as}\ \ \ell\to\infty.
\end{equation*}
Suppose first that $f\in\vartheta\big(U(\overline{D})\big)$. 
Then by type considerations
(and since $f$ is holomorphic in a neighborhood of $\overline{D}$) 
for any
$w\in D\degree$ and for any $\ell$ we have
$$
d_w\big( f(w)\Omega_0(\eta_\ell\degree)(w, z)\big) = 
\deebar_w\big( f(w)\Omega_0(\eta_\ell\degree)(w, z)\big) =
$$
$$
= f(w)\deebar_w\Omega_0(\eta_\ell\degree)(w, z) =
f(w) (\deebar_w\eta_\ell\degree)^n(w, z)
$$
Thus, applying Stokes' theorem on ${D}\degree$ 
 we find
\begin{equation*}
\int\limits_{w\in D\degree}\!\!\!\!\! f(w) (\deebar_w\eta_\ell\degree)^n(w, z)\, +\!\!\!
\int\limits_{w\in \bndry D}\!\!\!\!\!
f(w)\, j^*\Omega_0\left(\eta\degree_\ell\right)(w, z) = \!\!\!\!\!\!\!\!\!\!
\int\limits_{w\in D\,\cap\ \!\bndry (U_z (\bndry D))}\!\!\!\!\!\!\!\!\!\!\!\!\!\!\!
f(w)\, j^*\Omega_0\left(\eta\degree_\ell\right)(w, z) 
\end{equation*}
Letting $\ell\to\infty$ in the identity above we obtain
\begin{equation*}
\int\limits_{w\in D\degree}\!\!\!\!\! f(w) (\deebar_w\eta\degree)^n(w, z)\, +\!\!\!
\int\limits_{w\in \bndry D}\!\!\!\!\!
f(w)\, j^*\Omega_0\!\left(\eta\degree\right)(w, z) = \!\!\!\!\!\!\!\!\!\!
\int\limits_{w\in D\,\cap\ \!\bndry (U_z (\bndry D))}\!\!\!\!\!\!\!\!\!\!\!\!\!\!\!
f(w)\, j^*\Omega_0\left(\eta\degree\right)(w, z) 
\end{equation*}
Since $\eta\degree$ is generating at $z$, by basic property \ref{BP-2} this expression is reduced to
\begin{equation}\label{E:aux-2}
\int\limits_{w\in \bndry D}\!\!\!\!\!
f(w)\, j^*\Omega_0\!\left(\eta\degree\right)(w, z) = \!\!\!\!\!\!\!\!\!\!
\int\limits_{w\in D\,\cap\ \!\bndry (U_z (\bndry D))}\!\!\!\!\!\!\!\!\!\!\!\!\!\!\!
f(w)\, j^*\Omega_0\!\left(\eta\degree\right)(w, z) 
\end{equation}
But
\begin{equation*}
\eta\degree (w, z) =\left\{
\begin{array}{ll}
\eta (w, z),& \mbox{for}\ w\ \mbox{in an open neighborhood of}\ \bndry D\\
\\
\displaystyle \frac{\dee_w\beta}{\beta} (w, z),& \mbox{for}\ w\ \mbox{in an open neighborhood of}\ \bndry (U_z(\bndry D))
\end{array}
\right.
\end{equation*}
as a result, \eqref{E:aux-2} reads
$$
\int\limits_{w\in \bndry D}\!\!\!\!\!
f(w)\, j^*\Omega_0\!\left(\eta\right)(w, z)\ \  =
\int\limits_{w\in D\,\cap\ \!\bndry (U_z (\bndry D))}\!\!\!\!\!\!\!\!\!\!\!\!\!\!\!\!
f(w)\, \Sz_{BM}(w, z)
$$
On the other hand, 
$D\,\cap\ \!\bndry (U_z (\bndry D)) = \bndry V$ for a (smooth) open set $V\subset D$,
and using proposition \ref{P:repr-BM-CF-1}
we conclude that \eqref{E:repr-hol-bndry} holds in the case when $f\in\vartheta\big( U(\overline{D})\big)$.
To prove the conclusion in the general case: $f\in\vartheta (D)\,\cap\, C(\overline{D})$, we write $D=\{\rho(w)<0\}$,
so that $\rho(z)<0$ (since $z\in D$) and furthermore
\begin{equation}\label{D:D-k}
z\in D_k:=\left\{w\ \bigg|\ \rho(w)<-\frac{1}{k}\right\}\ \  \mbox{for any}\ k\geq k(z).
\end{equation}
But $D_k\subset D$ and so $f\in\vartheta\big(U(\overline{D}_k)\big)$; moreover
\begin{equation*}\label{E:bndry-D-k}
\bndry D_k\subset U_z \ \  \mbox{for}\ k =k(z)\ \ \mbox{sufficiently large}.
\end{equation*}
 Thus, \eqref{E:repr-hol-bndry}
  grants
\begin{equation*}\label{E:repr-hol-bndry-k-aux}
\int\limits_{w\in \bndry D_k}\!\!\!\!\!
f(w)\, j_k^*\Omega_0\!\left(\eta\right)(w, z) = f(z)\quad \mbox{for any}\ k\geq k(z)
\end{equation*}
where $j^*_k$ denotes the pullback under the inclusion $j_k:\bndry D_k\hookrightarrow\mathbb C^n$.

The conclusion now follows by letting $k\to\infty$.
\end{proof}
We remark that in the case when $\eta$ is the
 Bochner-Martinelli generating
 form $\eta:=\dee_w\beta/\beta$, proposition \ref{P:repr-hol-bndry} is simply a restatement of 
 proposition
	 \ref{P:repr-BM-CF-1}. However, since the coefficients
  of the Bochner-Martinelli CF form
   are nowhere holomorphic in the variable $z$, proposition \ref{P:repr-BM-CF-1} is of limited use in the investigation of the Cauchy-Szeg\"o and Bergman projections, and
proposition \ref{P:repr-hol-bndry}  will afford the use of more specialized choices of $\eta$.\\

\noindent The following reproducing formula is inspired by an idea of Ligocka \cite{L}.
\begin{Prop}\label{P:repr-hol-interior} 
With same hypotheses as in proposition \ref{P:repr-hol-bndry},  
 we have
\begin{equation*}\label{E:repr-hol-interior-special}
f(z) =
\frac{1}{(2\pi i)^n}\!\!\!\int\limits_{w\in D}\!\!\!\!
f(w)\, (\deebar_w\widetilde\eta\,)^n(w, z),\quad f\in \vartheta (D)\cap L^1(D)
\end{equation*}
for any $(1, 0)$-form\
$\widetilde\eta(\cdot, z)$ with coefficients in 
 $C^1(\overline{D})$ such that
\begin{equation}\label{E:extension}
j^*\Omega_0(\widetilde\eta) (\cdot, z) = j^*\Omega_0(\eta) (\cdot, z)\quad
\end{equation}
where $j^*$ denotes the pullback under the inclusion $j:\bndry D\hookrightarrow \mathbb C^n$.
\end{Prop}
Note that if one further assumes that the coefficients of $\widetilde\eta(\cdot, z)$ are in 
$C^2(D)\cap C^1(\overline{D})$ then, as a consequence of the fact that $\deebar\circ\deebar =0$, we have
$$\frac{1}{(2\pi i)^n}(\deebar_w\widetilde\eta\,)^n = \deebar_w\Omega_0(\widetilde\eta).$$
\begin{proof}
Fix $z\in D$ arbitrarily and let $\{\widetilde\eta_\ell\}_{\ell\in\mathbb N}\subset C^2_{1, 0}(\overline D)$ be such that
\begin{equation}\label{E:conv-C-1}
\| \widetilde\eta_\ell - \widetilde\eta(\cdot, z)\|_{C^1(\overline{D})}\to 0\quad \mbox{as}\ \ \ell\to\infty.
\end{equation}
Suppose first that $f\in\vartheta\big(U(\overline{D})\big)$. 
Applying Stokes' theorem to the $(n, n-1)$-form with coefficients in $C^1(\overline D)$
$$f\cdot\Omega_0(\widetilde\eta_\ell)
$$
we find
\begin{equation*}
\int\limits_{w\in D}\!\!\!\!
f(w)\, \deebar\Omega_0(\widetilde\eta_\ell)(w) =
\displaystyle
\int\limits_{w\in \bndry D}\!\!\!\!\!
f(w)\, j^*\Omega_0(\widetilde\eta_\ell)(w)\quad \mbox{for any}\ \ \ell.
\end{equation*}
On the other hand, since the coefficients of $\widetilde\eta_\ell$ are in $C^2(D)$, we have
$$
\deebar\Omega_0(\widetilde\eta_\ell) = \frac{1}{(2\pi i)^n}(\deebar\,\widetilde\eta_\ell\,)^n
\quad \mbox{for any}\ \ \ell
$$
and so the previous identity can be written as
\begin{equation*}
\frac{1}{(2\pi i)^n}\!\!\!\int\limits_{w\in D}\!\!\!\!
f(w)(\deebar\,\widetilde\eta_\ell\,)^n(w) =
\displaystyle
\int\limits_{w\in \bndry D}\!\!\!\!\!
f(w)\, j^*\Omega_0(\widetilde\eta_\ell)(w)\quad \mbox{for any}\ \ \ell.
\end{equation*}
Letting $\ell\to\infty$ in the identity above and  using \eqref{E:conv-C-1} we obtain
\begin{equation*}
\frac{1}{(2\pi i)^n}\!\!\!\int\limits_{w\in D}\!\!\!\!
f(w)(\deebar\,\widetilde\eta\,)^n(w, z) =
\displaystyle
\int\limits_{w\in \bndry D}\!\!\!\!\!
f(w)\, j^*\Omega_0(\widetilde\eta)(w, z).
\end{equation*}
Combining the latter with the hypothesis \eqref{E:extension} we obtain
\begin{equation*}
\frac{1}{(2\pi i)^n}\!\!\!\int\limits_{w\in D}\!\!\!\!
f(w)(\deebar_w\widetilde\eta\,)^n(w, z) =
\displaystyle
\int\limits_{w\in \bndry D}\!\!\!\!\!
f(w)\, j^*\Omega_0(\eta)(w, z) \ =\ f(z)
\end{equation*}
where the last identity is due to proposition \ref{P:repr-hol-bndry}.

If $f\in\vartheta (D)\cap L^1(D)$ then
$f\in\vartheta\big(U(\overline{D}_k)\big)$, where $D_k$ is as in \eqref{D:D-k}; moreover, $\bndry D_k\subset U_z \ \  \mbox{for any}\ \  k\geq k(z)$, so by the previous case we have
$$
f(z) = \int\limits_{w\in D_k}\!\!\!\!\!
f(w)\, (\deebar_w\widetilde\eta\,)^n(w, z)\quad \mbox{for any}\ \ k\geq k(z).
$$
The conclusion now follows by observing that
$$
 \int\limits_{w\in D_k}\!\!\!\!\!
f(w)(\deebar_w\widetilde\eta\,)^n(w, z)\to 
 \int\limits_{w\in D}\!\!\!\!
f(w)(\deebar_w\widetilde\eta\,)^n(w, z)
$$
 as $k\to\infty$, by the Lebesgue dominated convergence theorem.
\end{proof}
Note that the extension
$  \widetilde\eta (w, z) := 
    \chi_0(w,z)\eta (w, z)$,
   with $\chi_0$ as in \eqref{E:def-chi},
   satisfies a stronger condition than \eqref{E:extension}, namely the identity
  \begin{equation}\label{E:extension-stronger}
  \widetilde\eta(\cdot, z)=
   \eta(\cdot, z)
  \quad \mbox{for any}\ \ w\in U'_z(\bndry D).
  \end{equation}
  On the other hand,
  it will become clear in the sequel that this simple-minded
     extension is not an adequate
   tool for
  the investigation of the Bergman projection, and more ad-hoc constructions 
  are presented in Sections \ref{S:local-holom} and \ref{S:global-hol}.
  \bigskip

 \section{the role of pseudo-convexity}\label{S:convexity}
In order to obtain operators that satisfy the crucial condition (c) in Section \ref{S:CI-n}
one would need generating forms whose coefficients are holomorphic.
However, in contrast with the 
situation for
the planar case (where the Cauchy kernel plays the role of a universal generating form with holomorphic coefficient)
in higher dimension
there 
is a large class of
domains
$V\subset\mathbb C^n$ 
that cannot support 
generating forms
with holomorphic coefficients\footnote{much less a ``universal'' such form.}.
This dichotomy is related to
    the notion of 
   {\em domain of holomorphy}, that is, the property that for any boundary point $w\in \bndry V$ 
   there is a holomorphic function $f_w\in\vartheta (D)$ that cannot be continued holomorphically in a neighborhood of $w$.
    (Such notion is in turn equivalent
    to the notion of pseudo-convexity.)
  It is clear that every planar domain
  $V\subset\mathbb C$ is a domain of holomorphy, because in this case one may
  take $f_w(z):= (w-z)^{-1}$ where $w\in\bndry V$ has been fixed.
  On the other hand the following
     $$
  V=\{
   z\in\mathbb C^2\,|\, 1/2<|z|<1\}
     $$
 is a simple example of a smooth domain
  in 
 $\mathbb C^2$ that is not a domain
 of holomorphy; other classical examples are discussed e.g., in \cite[theorem II.1.1.]{Ra}. 
A necessary condition for the existence of a generating form $\eta$
   whose coefficients
 are  holomorphic in the sense described above
 is then that $V$ be a  domain of holomorphy. To prove 
 the necessity of such condition,
it suffices to  observe that as a consequence of \eqref{D:generating-1}
 and \eqref{D:generating-2} one has
 \begin{equation}\label{E:dom-holom}
  \sum\limits_{j=1}^{n} \eta_j(w, z)(w_j-z_j)\ =\  1\quad \text{for any}\ \ w\in \bndry V,\ \ z\in V.
 \end{equation}
 It is now clear that for each fixed $w\in \bndry V$,  
 at least one of the $\eta_j(w, z)$'s blows up as $z\to w$ (and it is well known that this is strong enough to ensure that $V$ be a domain of holomorphy).\\
 
In its current stage of development, the Cauchy-Fantappi\'e framework is most
effective in the analysis of
two particular
categories of pseudo-convex domains: these are the 
\emph{strongly pseudo-convex} domains and the related 
category  of
\emph{strongly $\mathbb C$-linearly convex} domains.
\begin{Def}\label{D:strongly-L-pscvx} 
We say that a domain
$D\subset\mathbb C^n$ is {\em strongly pseudo-convex} if $D$ is of class $C^2$ and if
{\em any} defining function $\rho$ for $D$ satisfies 
the following inequality
\begin{equation}\label{E:strongly-L-pscvx-II}
L_w(\rho)( \xi):= \sum\limits_{j, k=1}^n\frac{\dee^2\rho (w)}{\dee\zeta_j\dee\overline{\zeta}_k}\, 
\xi_j\overline{\xi}_k>0\ \mbox{for any}\ w\in\bndry D,\ 
\xi\in T_w^{\mathbb C}(\bndry D)
\end{equation}
where $ T_w^{\mathbb C}$ denotes the {\em complex tangent space to $\bndry D$ at $w$}, 
namely
\begin{equation*}
T^{\mathbb C}_w(\bndry D) =\{\xi\in\mathbb C^n\ |\ \langle\dee\rho (w), \xi\rangle=0\},
\end{equation*}
 see \cite[proposition II.2.14]{Ra}. 
 \end{Def}
 
 If $D$ is of class $C^k$ with $k\geq 1$, and if $\rho_1$ and $\rho_2$ are two distinct defining functions for
$D$, it can be shown that there is a positive function $h$ of class $C^{k-1}$ in a neighborhood $U$ of the boundary of $D$, such that
\begin{equation*}\label{E:def-functs}
\rho_1(w) = h(w)\rho_2(w),\ \ w\in U,
\end{equation*}
and
\begin{equation}\label{E:def-functs}
\nabla\!\rho_1(w) = h(w)\,\nabla\!\rho_2(w)\quad \mbox{for any}\ w\in U\cap \bndry D,
\end{equation}
 see \cite[lemma II.2.5]{Ra}. 
 As a consequence of \eqref{E:def-functs}, if condition \eqref{E:strongly-L-pscvx-II}
 is satisfied by \emph{one} defining function then it will be satisfied by \emph{every} defining function.
  The hermitian
form $L_w(\rho)$ defined by \eqref{E:strongly-L-pscvx-II} is called
 {\em the Levi form, or complex Hessian, of $\rho$ at $w$}.
 We remark that in fact there is a ``special'' defining function $\rho$ for $D$ that is strictly plurisubharmonic on a neighborhood
$U$ of $\overline{D}$, that is 
\begin{equation*}\label{E:strongly-L-pscvx-I}
L_w(\rho)( \xi)
>0\quad \mbox{for any}\ w\in U\ \mbox{and any}\   \xi\in\mathbb C^n\setminus\{0\},
\end{equation*}
see \cite[proposition II.2.14]{Ra}, and we will assume throughout the sequel that $\rho$ satisfies this stronger condition.

We should point out that there is another 
 notion of strong pseudo-convexity that 
  includes the domains of definition \ref{D:strongly-L-pscvx} as a subclass
  (this notion does not require the gradient of $\rho$ to be non-vanishing on $\bndry D$);
  within this more general context, the domains of definition \ref{D:strongly-L-pscvx} are sometimes referred to as
``strongly Levi-pseudo-convex'', see \cite[\S II.2.6 and II.2.8]{Ra}.\\

\begin{Def}\label{D:strongly-C-lin-cvx} We say that
$D\subset \mathbb C^n$
is  {\em strongly $\mathbb C$-linearly convex}  if $D$ is of class $C^1$ and if any
defining function
 for $D$ satisfies this
  inequality:
\begin{equation*}\label{E:strongly-C-lin-cvx}
|\langle\dee\rho(w), w-z\rangle|\geq C|w-z|^2\quad\mbox{for any}\ w\in\bndry D,\ 
z\in \overline{D}.
\end{equation*}
\end{Def}

We we call those domains that satisfy the following, weaker condition
\begin{equation*}\label{E:strictly-C-lin-cvx}
|\langle\dee\rho(w), w-z\rangle|>0\quad\mbox{for any}\ w\in\bndry D\ 
\mbox{and any}\ z\in \overline{D}\setminus \{w\}
\end{equation*}
 {\em strictly $\mathbb C$-linearly convex}.  This condition is related to certain separation properties of the domain from its complement by (real or complex) hyperplanes, see \cite{APS}, \cite[IV.4.6]{Ho}: that this must be so is a consequence of
  the assertion that, for $w$ and $z$ as in \eqref{E:strongly-C-lin-cvx},
 the quantity $|\langle\dee\rho(w), w-z\rangle|$ is comparable to the Euclidean distance of $z$ to the complex tangent space $T^{\mathbb C}_w(\bndry D)$; we leave the verification of
this assertion as an exercise for the reader.

 It is not difficult to check that
$$
D:=\{z\in\mathbb C^n\ |\ \mbox{Im}\,z_n>(|z_1|^2+\cdots+ |z_{n-1}|^2)^2\,\}
$$
is strictly, but not strongly,
 $\mathbb C$-linearly convex.

\begin{Lem}\label{L:strictly-C-lin-cvx}
If $D$ is strictly $\mathbb C$-linearly convex then for any $z\in D$ there is an open set 
$U_z\subset \mathbb C^n\setminus\{z\}$ such that
$\bndry D\subset U_z$ and inequality \eqref{E:strictly-C-lin-cvx} holds for any $w$ in $U_z$.
Furthermore, if $D$ is strongly $\mathbb C$-linearly convex then the improved inequality \eqref{E:strongly-C-lin-cvx} 
will hold for any $w\in U_z$.
\end{Lem}
\begin{proof}
Suppose that $D$ is strictly $\mathbb C$-linearly convex and fix $z\in D$. By the continuity of the function $h(\zeta):= |\langle\dee\rho(\zeta), \zeta- z\rangle|$,
if \eqref{E:strictly-C-lin-cvx} holds at $w\in \bndry D$ then there is an open neighborhood 
$U_z(w)$ such that $h(\zeta)>0$ for any $\zeta\in U_z(w)$
and so
 we have that $h(\zeta)>0$ whenever 
$$
\zeta\in U_z:= \bigcup\limits_{w\in\bndry D}U_z(w).
$$
It is clear that $\bndry D\subset U_z$; furthermore, since $h(z)=0$ then 
$U_z(w)\subset\mathbb C^n\setminus\{z\}$ for any $w\in\bndry D$ and so $U_z\subset\mathbb C^n\setminus\{z\}$.

If $D$ is strongly $\mathbb C$-linearly convex then the conclusion will follow by considering the function $h(\zeta) :=
|\langle\dee\rho(\zeta), \zeta- z\rangle|- C|\zeta- z|^2$.
\end{proof}

\begin{Rmrk}\label{R:classical-cvx}
We recall that  in the classical definition of {\em strong} (resp.  {\em strict}) {\em convexity}, 
  the quantity $|\langle\dee\rho(w), w-z\rangle|$ in the left-hand side of \eqref{E:strongly-C-lin-cvx} 
(resp. \eqref{E:strictly-C-lin-cvx}) is replaced by $\mbox{Re}\langle\dee\rho(w), w-z\rangle$:
it follows that any strongly (resp. strictly) convex domain is  indeed strongly (resp. strictly) $\mathbb C$-linearly convex, but the converse is in general not true. It is clear that strongly
(resp. strictly) convex domains satisfy a version of lemma \ref{L:strictly-C-lin-cvx}.
\end{Rmrk}

\begin{Lem}\label{L:strongly-L-pscvx}
 Any strongly $\mathbb C$-linearly convex domain of class $C^2$ is
strongly pseudo-convex.
\end{Lem}
The key point in the proof of this lemma is the 
observation that, as a consequence of \eqref{E:strongly-C-lin-cvx}, the real tangential Hessian of any defining function for a domain as in lemma \ref{L:strongly-L-pscvx}
is positive definite when restricted to the complex tangent space $T^{\mathbb C}_w(\bndry D)$
(viewed as a vector space over the \emph{real} numbers).
The converse of lemma \ref{L:strongly-L-pscvx} is not true: we leave as an exercise for the reader to verify that the following (smooth) domain
$$
D:=\{z=(z_1, z_2)\in\mathbb C^2\ |\ \mbox{Im}z_2>2(\mbox{Re}\,z_1)^2-(\mbox{Im}\,z_1)^2\}
$$
 is strongly pseudo-convex but not strongly $\mathbb C$-linearly convex.\\

In closing this section we remark that while the designations ``strongly'' and  ``strictly'' indicate
distinct families of $\mathbb C$-linearly convex domains (and of convex domains), for
pseudo-convex domains there is no 
 such distinction, and in fact in the literature the terms 
``\emph{strictly} pseudo-convex'' and ``\emph{strongly} pseudo-convex'' are often interchanged: this is because the positivity condition \eqref{E:strongly-L-pscvx-I} implies the seemingly stronger inequality
\begin{equation}\label{E:strongly-L-pscvx-III}
L_w(\rho, \xi) \geq c_0|\xi|^2 \ \ \mbox{for any} \ w\in U'\ \mbox{and for any}\ \xi\in\mathbb C^n.
\end{equation}
 Indeed, if \eqref{E:strongly-L-pscvx-I} holds
 then the function $\gamma (w) :=\min\{L_w(\rho, \xi)\ |\ |\xi|=1\}$ is positive, and by bilinearity it follows that $L_w(\rho, \xi) \geq \gamma (w) |\xi|^2$ for any $\xi\in\mathbb C^n$; since $\rho$ is of class $C^2$ (and $D$ is bounded) we may further take 
the minimum of  $\gamma (w)$ over, say, $w\in U'\subset U$ and
thus obtain  \eqref{E:strongly-L-pscvx-III}, 
see \cite[II.(2.26)]{Ra}.
\bigskip

  \section{locally holomorphic kernels}
    \label{S:local-holom}
  
  A first step in the study of the Bergman and Cauchy-Szeg\"o projections 
  is the construction of integral operators with kernels given by Cauchy-Fantappi\'e forms that are
 (at least)
{\em locally} holomorphic in $z$, that is for $z$ in a neighborhood of each (fixed) $w$:
 it is at this juncture that the 
 notion of strong pseudo-convexity
 takes center stage. 
In this section we show how to construct such operators in the case when
$D$ is a bounded, strongly pseudo-convex domain, and we then proceed to
prove the reproducing property.

To this end, we fix a
strictly plurisubharmonic defining function for $D$; that is, we fix
$$
\rho:\mathbb C^n\to\mathbb R,\quad \rho\in C^2(\mathbb C^n)
$$
such that $D=\{\rho<0\}$;\ $\nabla\rho(w)\neq 0$ for any $w\in\bndry D$ and 
\begin{equation*}\label{E:strictly-plush}
L_w(\rho, w-z) 
\,\geq\, 2c_0\,|w-z|^2,\quad w, z\in\mathbb C^n
\end{equation*}
where $L_w$ denotes the Levi form for $\rho$, see \eqref{E:strongly-L-pscvx-II} and
 \eqref{E:strongly-L-pscvx-III}.
Consider the {\em Levi polynomial of $\rho$  at $w$}:
\begin{equation*}\label{D:levi-pol}
\Delta (w, z) := \langle\dee\rho (w), w-z\rangle -
\frac{1}{2}\sum\limits_{j, k=1}^n
\frac{\dee^2\rho (w)}{\dee\zeta_j\dee\zeta_k}(w_j-z_j)(w_k-z_k)
\end{equation*}
\begin{Lem}\label{L:levi-pol}
Suppose $D=\{\rho(w)<0\}$ is bounded and strongly pseudo-convex. Then, 
 there is
 $\widetilde\epsilon_0 = \widetilde\epsilon_0(c_0)>0$ such that
\begin{equation*}\label{E:levi-pol}
2\mathrm{Re}\,\Delta (w, z) \geq \rho(w)-\rho(z) + c_0|w-z|^2
\end{equation*}
whenever $w\in D_{\!c_0}=\{w\,|\,\rho(w)<c_0\},$ and $z\in \overline{\mathbb B_{\,\widetilde\epsilon_0}(w)}$.
\end{Lem}
Here $c_0$ is as in \eqref{E:strictly-plush}. We leave the proof of this lemma, along with the 
corollary below, as an exercise for the reader.
Now let $\chi_1(w, z)$ be a smooth cutoff function such that
\begin{equation}\label{D:chi-0}
\chi_1(w, z) =
\left\{\begin{array}{rcl}
1,&\mbox{if}&|w-z|<\widetilde\epsilon_0/2\\
0,&\mbox{if}&|w-z|>\widetilde\epsilon_0
\end{array}
\right.
\end{equation}
where $\widetilde\epsilon_0$ is as in lemma \ref{L:levi-pol} and set
\begin{equation}\label{D:g-tilde}
\g (w, z) = \chi_1(w, z)\Delta (w,z) +(1-\chi_1(w, z))|w-z|^2,\ \ w, z\in \mathbb C^n.
\end{equation}
\smallskip

\begin{Lem}\label{L:denom-aux}
Suppose $D=\{\rho(w)<0\}$ is strongly pseudo-convex and of class $C^2$. Then, there is 
$\widetilde\delta_0=\widetilde\delta_0(\widetilde\epsilon_0, c_0)>0$ such that
\begin{equation*}
2\mathrm{Re}\,\g (w, z) \geq 
\left\{\begin{array}{rcl}
\rho (w) - \rho (z) + c_0|w-z|^2,&\mathrm{if}& |w-z|\leq \widetilde\epsilon_0/2\\
\\
\rho(w) + 2\widetilde\delta_0,&\mathrm{if}&\widetilde\epsilon_0/2\leq |w-z|<\widetilde\epsilon_0\\
\\
\widetilde\epsilon_0^2,&\mathrm{if}&|w-z|>\widetilde\epsilon_0
\end{array}
\right.
\end{equation*}
whenever 
\begin{equation}\label{E:D-c-0}
w\in D_{\!c_0}=\{w\ |\ \rho(w)<c_0\}
\end{equation}
 and 
$$z\in D_{2\widetilde\delta_0}=\{w\ |\ \rho(w)<2\widetilde\delta_0\}.$$
\end{Lem}
\begin{proof}
It suffices to choose $0<\widetilde\delta_0<c_0\widetilde\epsilon_0^{\, 2}/16$: the desired inequalities then follow from
lemma \ref{L:levi-pol}.
\end{proof}
\begin{Cor}\label{C:gen-form-pscvx} Let $D=\{\rho(w)<0\}$ be a bounded, strongly pseudo-convex domain.
 Let
\begin{equation*}
\Delta_j(w, z)\ :=\ \frac{\dee\rho}{\dee\zeta_j}(w) -\frac{1}{2}
\sum\limits_{k=1}^n\frac{\dee^2\rho (w)}{\dee\zeta_j\dee\zeta_k}(w_k-z_k),\quad j=1,\ldots, n,
\end{equation*}
Define
\begin{equation*}\label{D:gen-form-pscvx}
\!\, \eta_j(w, z)\ := \ 
\frac{1}{\g(w, z)}\bigg(\chi_1(w, z)\Delta_j(w, z) + (1-\chi_1(w, z))(\overline{w}_j-\overline{z}_j)\!\bigg)
\end{equation*}
where $\chi_1$ and $\g$ are as in \eqref{D:chi-0} and \eqref{D:g-tilde}, and set
$$
\!\,\eta (w, z) :=\sum\limits_{j=1}^n\!\,\eta_j(w, z)\,dw_j \quad \mbox{for}\ \ 
(w, z)\, \in\ D_{\!c_0}\times D
$$
with $D_{\!c_0}$ as in  \eqref{E:D-c-0}.
Then
we have that $\!\,\eta(w, z)$ is a generating form for $D$, and one 
 may take for $U_z$ in definition \ref{D:gen-form} the set
\begin{equation}\label{D:U-z}
U_z:=\bigg\{w\ |\ \max\{\rho (z), -\widetilde\delta_0\}<\rho(w)<\min\{|\rho (z)|, c_0\}\,\bigg\}.
\end{equation}
\end{Cor}
Note, however, that the coefficients of $\!\,\eta$ in this construction are only continuous in the variable $w$ 
 and so the Cauchy-Fantappi\'e form
$\Omega_0(\!\,\eta)$ cannot be defined for such 
$\!\,\eta$ because doing so would require differentiating the coefficients of $\!\,\eta$ with respect to $w$, see \eqref{D:CF-0}. For this reason,
proceeding as in  \cite{Ra}, we refine
the previous construction as follows. For $\widetilde\epsilon_0$ as in
 lemma \ref{L:levi-pol} and for any $0<\epsilon<\widetilde\epsilon_0$,  we let $\tau_{j, k}^\epsilon\in C^\infty(\mathbb C^n)$ be such that
\begin{equation*}
\max\limits_{w\in\overline{D}}
\left|\frac{\dee^2\rho(w)}{\dee\zeta_j\dee\zeta_k} - \tau_{j, k}^\epsilon(w)\right| < \epsilon,\qquad j, k=1,\ldots, n
\end{equation*}
 We now define the following quantities:
 \begin{equation}\label{E:smooth}
\Delta_j^\epsilon(w, z)\ :=\ \frac{\dee\rho}{\dee\zeta_j}(w) -\frac{1}{2}
\sum\limits_{k=1}^n\tau_{j, k}^\epsilon(w)(w_k-z_k),\quad j=1,\ldots, n;
\end{equation}
  \begin{equation*}
 \Delta^{\!\epsilon}(w, z) := \sum\limits_{j=1}^n \Delta_j^\epsilon(w, z)\,(w_j-z_j);
\end{equation*}
 and, for $\chi_1$ as in \eqref{D:chi-0}:
\begin{equation}\label{D:g-eps}
g^{\epsilon}(w, z):= \chi_1(w, z)\Delta^{\!\epsilon} (w,z) +(1-\chi_1(w, z))|w-z|^2;
\end{equation}
\begin{equation*}
\eta_j^{\epsilon}(w, z)\ := \ 
\frac{1}{g^\epsilon(w, z)}\bigg(\chi_1(w, z)\Delta_j^{\!\epsilon}(w, z) + (1-\chi_1(w, z))(\overline{w}_j-\overline{z}_j)\!\bigg)
\end{equation*}
and finally
\begin{equation*}
\eta^\epsilon (w, z) :=\sum\limits_{j=1}^n\eta_j^\epsilon(w, z)\,dw_j\, .
\end{equation*}
\begin{Lem}\label{L:gen-form-pscvx-eps} Let $D=\{\rho(w)<0\}$ be a bounded strongly pseudo-convex domain.
Then, there is $\epsilon_0 = \epsilon_0(c_0)>0$
  such that for any $0<\epsilon<\epsilon_0$ and for any $z\in D$, we have that
    $\eta^{\epsilon}(w, z)$ defined as above is generating at $z$ relative to $D$ with an open set $U_z$ (see definition \ref{D:gen-form}) that does not depend on $\epsilon$. Furthermore, we have that  for each (fixed) $z\in D$
   the coefficients of $\eta^{\epsilon}(\cdot, z)$ are in $C^1(U_z)$.
\end{Lem}
\begin{proof} We first observe that $\Delta^{\!\epsilon}$ can be expressed in terms of the Levi polynomial $\Delta$, as follows
\begin{equation*}
\Delta^{\!\epsilon}(w, z) := \Delta (w,z) +
\frac{1}{2}\sum\limits_{j, k=1}^n\!\left(\frac{\dee^2\rho(w)}{\dee\zeta_j\dee\zeta_k} - \tau_{j, k}^\epsilon(w)\!\right)\!\!(w_j-z_j)(w_k-z_k)
\end{equation*}
and so by lemma \ref{L:levi-pol} we have
\begin{equation*}\label{E:levi-pol}
2\mathrm{Re}\,\Delta^{\!\epsilon} (w, z) \geq \rho(w)-\rho(z) + c_0|w-z|^2
\end{equation*}
for any
$$
0<\epsilon<\epsilon_0:=\min\{\widetilde\epsilon_0,\, 2c_0/n^2\}
$$
whenever $w\in D_{\!c_0} =\{\rho(w)<c_0\}$ and $z\in\overline{\mathbb B_{\epsilon_0}(w)}$. Proceeding as in the proof of lemma \ref{L:denom-aux} we then find that
\begin{equation*}
2\mathrm{Re}\,g^{\epsilon} (w, z) \geq 
\left\{\begin{array}{rcl}
\rho (w) - \rho (z) + c_0|w-z|^2,&\mathrm{if}& |w-z|\leq \epsilon_0/2\\
\\
\rho(w) + \mu_0,&\mathrm{if}&\epsilon_0/2\leq |w-z|<\widetilde\epsilon_0\\
\\
\widetilde\epsilon_0^{\ \!2},&\mathrm{if}&|w-z|\geq\widetilde\epsilon_0
\end{array}
\right.
\end{equation*}
for any $0<\epsilon<\epsilon_0$ whenever 
$$w\in D_{\!c_0}=\{w\ |\ \rho(w)<c_0\}$$ and 
$$z\in D_{\mu_0}=\{w\ |\ \rho(w)<\mu_0\}$$ as soon as we
choose 
\begin{equation}\label{D:mu_0}
0<\mu_0<c_0\epsilon_0^{\, 2}/8.
\end{equation}
 We then define the open set $U_z\subset\mathbb C^n\setminus\{z\}$ as in \eqref{D:U-z} but now with  $\delta_0$ in place of $\widetilde\delta_0$ (note that  $U_z$ does not depend on $\epsilon$). Then, proceeding  as in the proof of
corollary \ref{C:gen-form-pscvx} we find that 
\begin{equation*}\label{E:lower-bd-g-eps}
\inf\limits_{w\in U_z}\mbox{Re}\,g^{\epsilon}(w, z)> 0\quad \mbox{for any} \ \
0<\epsilon<\epsilon_0.
\end{equation*}
From this it follows that $\eta^{\,\epsilon}$ is a generating form for $D$; it is clear
from \eqref{E:smooth} that the coefficients of $\eta^{\,\epsilon}$ are in $C^1(U_z)$.
\end{proof}
 lemma \ref{L:gen-form-pscvx-eps} shows that $\eta^{\epsilon}$ satisfies the hypotheses of 
 proposition \ref{P:repr-hol-bndry}; as a consequence we obtain the following results:
\begin{Prop}\label{P:repr-hol-eps-bndry} Let $D$
be a bounded strongly pseudo-convex domain. Then, for any $0<\epsilon<\epsilon_0$ 
we have
\begin{equation*}\label{E:repr-hol-eps-bndry}
f(z) =\!\!
\int\limits_{w\in \bndry D}\!\!\!\!\!
f(w)\, j^*\Omega_0\!\left(\eta^{\epsilon}\right)\!(w, z)\quad \mbox{for any}\ 
f\in \vartheta (D)\cap C(\overline{D}),\ 
z\in D
\end{equation*}
where $\epsilon_0$ and $\eta^\epsilon$ are as in lemma \ref{L:gen-form-pscvx-eps}.
\end{Prop}
\begin{Prop}\label{P:repr-hol-eps-interior} Let $D=\{\rho(w)<0\}$ be a bounded strongly pseudo-convex domain. Let 
\begin{equation*}\label{D:h-eps}
\widetilde\eta^{\,\epsilon} (w, z)\ :=\ \frac{g^\epsilon(w, z)}{g^\epsilon(w, z)-\rho(w)}\, \eta^{\,\epsilon}(w, z),\quad w\in \overline D,\ \ z\in D.
\end{equation*}
where $\eta^\epsilon$ is as in lemma \ref{L:gen-form-pscvx-eps}.
 Then, for any $0<\epsilon<\epsilon_0$ 
we have
\begin{equation*}\label{E:repr-hol-eps-interior}
f(z) =
\frac{1}{(2\pi i)^n}\!\!\!
\int\limits_{w\in D}\!\!\!\!
f(w)(\deebar_w\widetilde\eta^{\,\epsilon})^{n}(w, z)\quad 
\mbox{for any}\ f\in \vartheta (D)\cap L^1(D),\ z\in D.
\end{equation*}
\end{Prop}
\begin{proof}
We claim that $\widetilde\eta^{\,\epsilon}$ 
satisfies the hypotheses of  proposition \ref{P:repr-hol-interior} for any $0<\epsilon<\epsilon_0$.
Indeed, proceeding as in the proof of lemma \ref{L:gen-form-pscvx-eps} we find that
\begin{equation*}\label{E:ineq-closure}
\mbox{Re}\bigg(g^\epsilon(w, z)-\rho(w)\bigg)>0\quad \mbox{for any}\ w\in\overline{D},\ 
\mbox{ for any}\ 
\ z\in D
\end{equation*}
and for any $0<\epsilon<\epsilon_0$;
from this it follows that 
 $$
 \widetilde\eta^{\,\epsilon}(\cdot, z)\in\ C^1_{1, 0}(\overline{D})\quad \mbox{for any}\ 0<\epsilon<\epsilon_0.
 $$
Moreover, as a consequence of basic property \ref{BP-1} we have
$$
\Omega_0(\widetilde\eta^{\,\epsilon})(\cdot, z) = 
\left(\frac{g^\epsilon(\cdot, z)}{g^\epsilon(\cdot, z)-\rho(\cdot)}\right)^{\!\!n}
\!
\Omega_0(\eta^{\,\epsilon})(\cdot, z)\quad \mbox{for any}\ 0<\epsilon<\epsilon_0,
$$
but this grants
$$
j^*\Omega_0(\widetilde\eta^{\,\epsilon})(\cdot, z)  =
j^*\Omega_0(\eta^{\,\epsilon})(\cdot, z) \quad \mbox{for any}\ 0<\epsilon<\epsilon_0.
$$
The conclusion now follows from proposition \ref{P:repr-hol-interior}.
\end{proof}
\bigskip

\section{Correction terms}\label{S:correction}
 propositions \ref{P:repr-hol-eps-bndry} and \ref{P:repr-hol-eps-interior} have a fundamental limitation: it is that these propositions  employ kernels, namely 
  $j^*\Omega_0(\eta^{\,\epsilon})(w, z)$ and 
$(\deebar_w\widetilde\eta^{\,\epsilon})^{n}(w, z)$,
 that are only \emph{locally} holomorphic as functions of $z$, that is, they are holomorphic only for $z\in 
 \mathbb B_{\epsilon_0/2}(w)$. In this section we address this issue by constructing for each of these kernels a ``correction'' term obtained by solving an ad-hoc $\deebar$-problem in the $z$-variable.
 
Throughout this section we shift our focus from the $w$-variable to $z$, that is: we fix $w\in \overline{D}$, we regard $z$ as a variable and we define the ``parabolic'' region

\begin{equation*}
\mathcal P_w := \{z\ |\ \rho (z) + \rho (w)< c_0|w-z|^2\}.
\end{equation*}
\smallskip

\noindent 
 The region $\mathcal P_w$ has the following properties:

\begin{equation*}
w\in\overline{D}\ \Rightarrow \ D\subset \mathcal P_w\ ;
\end{equation*}
\begin{equation*}
w\in\bndry D \ \Rightarrow \ z:= w\in \bndry  \mathcal P_w.
\end{equation*}
\smallskip

\noindent As a consequence of these properties we have that
\begin{equation*}\label{E:intersection}
\mathcal P_w \cap \mathbb B_{\epsilon_0/2}(w) \neq \emptyset
\end{equation*}
\begin{Lem}\label{L:parabolic}
Let $D=\{z\ |\ \rho(z)<0\}$ be a bounded strongly pseudo-convex domain.
Then, there is $\mu_0 = \mu_0(c_0)>0$ such that 
\begin{equation}\label{E:mu-0}
D_{\mu_0}=\{z\ |\ \rho(z)<\mu_0\}\subset \mathcal P_w\cup \mathbb B_{\epsilon_0/2}(w)
\end{equation}
for any (fixed) $w\in \overline{D}$.
 Furthermore,
 there is a bounded strongly pseudo-convex $\Omega$ of class 
 $C^\infty$ such that 
 \begin{equation*}
 D_{\mu_0/2}=\{z\ |\ \rho(z)<\mu_0/2\}\ \subset\  \Omega\  \subset\  D_{\mu_0}=\{z\ |\ \rho(z)<\mu_0\}
 \end{equation*}
 where $\mu_0>0$ is as in \eqref{E:mu-0}.
 \end{Lem}
 \begin{figure}[h!]
 \centering 
 \includegraphics[width=0.5 \textwidth]
{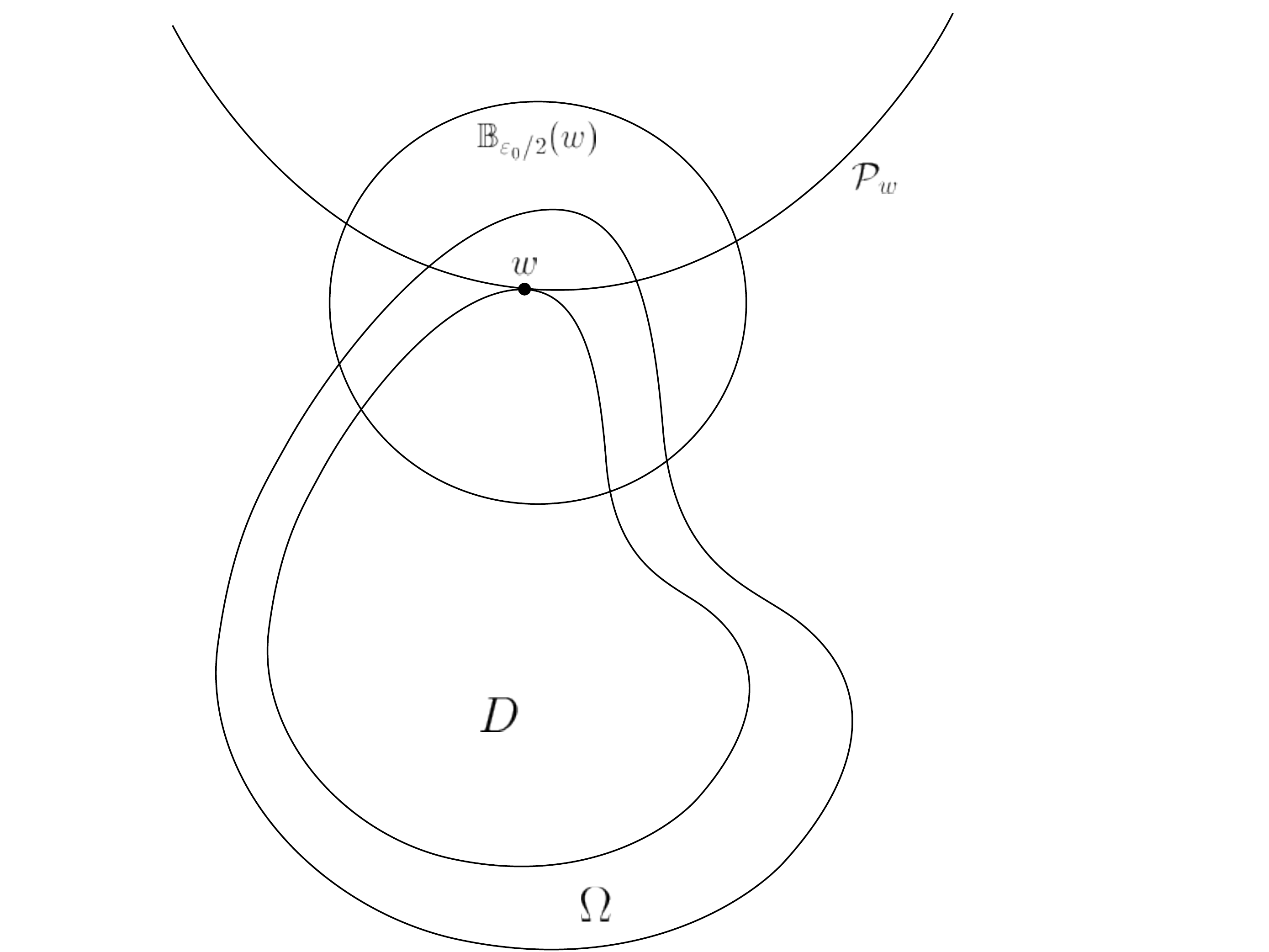}
\caption{The region $\mathcal P_w$ in the case when $w\in\bndry D$.}
\label{fig:Figure1}
\end{figure}
  
\emph{Proof of lemma \ref{L:parabolic}}. 
For the first conclusion, we claim that it suffices to choose $\mu_0 =\mu_0(c_0)$ as in
\eqref{D:mu_0}. Indeed, given
 $z\in D_{\mu_0}$, if $|w-z|\geq \epsilon_0/2$ then $\rho(z)\leq c_0|w-z|^2/2$ and since 
 $\rho(w)\leq 0$ (as $w\in \overline{D}$) it follows that $z\in \mathcal P_w$. On the other hand, 
 if $|w-z|< \epsilon_0/2$ then of course $z\in \mathbb B_{\epsilon_0/2}(w)$.

To prove the second conclusion note that, since $\rho$ (the defining function of $D$)
is of class $C^2$ and is strictly plurishubharmonic in a neighborhood of $\overline{D}$, there is 
 $\widetilde\rho\in C^\infty (U(\overline{D}))$ such that
 
$$
\|\widetilde\rho -\rho\|_{C^2(U(\overline{D}))}\leq \mu_0/8
$$
and
$$
L_z(\widetilde\rho, \xi)>0\quad \mbox{for any}\ z\in U'(\overline{D})\ \ \mbox{and for any }\ \ \xi\in\mathbb C^n,
$$
see \eqref{E:strongly-L-pscvx-II} and \eqref{E:strongly-L-pscvx-I}. Define
\begin{equation*}\label{D:Omega}
\Omega := \left\{z\ \bigg|\ \widetilde\rho(z) - \frac{3\mu_0}{4}<0\right\}
\end{equation*}
Then $\Omega$ is smooth and strongly pseudo-convex; we leave it as an exercise
 for the reader to verify that $\Omega$ satisfies the desired inclusions: $D_{\mu_0/2}\subset\Omega\subset D_{\mu_0}$.
\qed
\smallskip

lemma \ref{L:parabolic} shows that (the smooth and strongly pseudo-convex domain) $\Omega$ has the following properties, see Figure \ref{fig:Figure1}:

$$
\overline{D}\subset \Omega,\quad\mbox{and}\quad  \overline{\Omega}\ \subset
 \mathcal P_w\cup \mathbb B_{\epsilon_0/2}(w),\quad \mbox{for every}\ \ w\in\overline{D}.
$$

\noindent We now set up two $\overline{\partial}$-problems on $\Omega$.
 For the first $\deebar$-problem, we begin by observing that  if $w$ is in $\bndry D$ and $z$
 is in $\mathcal P_w$ then $\mbox{Re}\,g^\epsilon (w, z)>0$
 (that  this must be so can be seen from the inequalities for Re\,$g^{\epsilon}(w, z)$ that were obtained in the proof of lemma \ref{L:gen-form-pscvx-eps}), and so the coefficients of $\eta^{\,\epsilon}(w, \cdot)$ are in 
   $C^\infty(\mathcal P_w)$ whenever $w\in\bndry D$.
 Now fix $w\in \bndry D$ arbitrarily
 and denote by $\H(w, z) = \H_\epsilon (w, z)$ the following double form, which is of type $(0, 1)$ in $z$, and of type $(n,n-1)$ 
in $w$ 
\begin{equation} \label{E:def-d-bar-data-bndry}
\H(w, z)\ =\
\Bigg\{\!\!
\begin{array}{lll}
- \overline{\partial}_z \Omega_0(\eta^{\,\epsilon}) (w, z),
  &\text{if}\ z \in \mathcal{P}_w\\
 \quad &\\
 \quad\  0, &\text{if}\  z\in \mathbb B_{\epsilon_0/2}(w)
 \end{array}
\end{equation}
Now for each fixed $w\in\bndry D$, the coefficients 
of $\Omega_0(w, z)$ are holomorphic in $z$ for $z \in \mathbb B_{\epsilon_0/2}(w)$ 
and so $\H(w, z)$ is defined consistently in 
$\mathcal{P}_w \cup \mathbb B_{\epsilon_0/2}(w)$. It is also clear
that $\H(w, z)$ is $C^\infty$ for $z \in \mathcal{P}_w \cup \mathbb{B}_{\epsilon_0/2}(w)$, and as such it depends continuously on $w \in \bndry D$. Moreover we have that $\overline{\partial}_z \H(w, z) = 0$, for $z \in \mathcal{P}_w \cup \mathbb{B}_{\epsilon_0/2}(w) , \ w \in \bndry D$.
\smallskip

 For the second $\deebar$-problem,  we begin by observing that  if $w$ is in $\overline D$ and $z$
 is in $\mathcal P_w$ then $\mbox{Re}\left(g^\epsilon (w, z)-\rho (w)\right)>0$
 (that  this must be so can again be seen from the inequalities for Re\,$g^{\epsilon}(w, z)$ in the proof of lemma \ref{L:gen-form-pscvx-eps}), and so the coefficients of $\widetilde\eta^{\,\epsilon}(w, \cdot)$ are in 
   $C^\infty(\mathcal P_w)$ whenever $w\in\overline D$.
Fixing $w \in \overline{D}$ arbitrarily, we denote by 
 $\F(w, z) = \F_\epsilon (w, z)$ the following double form, which is of type $(0, 1)$ in $z$, and of type $(n,n)$ 
in $w$ 
\begin{equation*} \label{E:def-d-bar-data-interior}
\F(w, z)\ =\
\Bigg\{\!\!
\begin{array}{lll}
- \overline{\partial}_z (\deebar_w\widetilde\eta^{\,\epsilon})^n (w, z),
  &\text{if}\ z \in \mathcal{P}_w\\
 \quad &\\
 \quad\  0, &\text{if}\  z\in \mathbb B_{\epsilon_0/2}(w)
 \end{array}
\end{equation*}
Now for each fixed $w\in\overline{D}$, the coefficients 
of $\widetilde\eta^{\,\epsilon}(w, z)$ are holomorphic 
in $z$ for $z \in \mathbb B_{\epsilon_0/2}(w)$ 
and so $\F(w, z)$ is defined consistently in 
$\mathcal{P}_w \cup \mathbb B_{\epsilon_0/2}(w)$. It is also clear 
that $\F(w, z)$ is $C^\infty$ for $z \in \mathcal{P}_w \cup 
 \mathbb B_{\epsilon_0/2}(w)$, and as such it depends continuously on $w \in \overline{D}$. Moreover we have that
 $\overline{\partial}_z \F(w, z) = 0$, for $z \in \mathcal{P}_w \cup  \mathbb B_{\epsilon_0/2}(w) , \ w \in \overline{D}$.
\smallskip

Now let $\s= \s_z$ be the solution operator, giving the normal solution of the problem $\overline{\partial} u = \alpha$ in $\Omega$, via the $\overline{\partial}$-Neumann problem, so that $u = \s(\alpha)$ satisfies the above whenever $\alpha$ is a $(0, 1)$-form with $\overline{\partial} \alpha = 0$. We set
 \begin{equation}\label{E:def-correction-szego}
 \Sz^2_\epsilon (w, z) = \s_z (\H(w, \cdot) ),\quad w\in\bndry D
 \end{equation}
and
 \begin{equation*}\label{E:def-correction-bergman}
 \\B^2_\epsilon (w, z) = \s_z (\F(w, \cdot) ),\quad w\in\overline{D}.
 \end{equation*}
  Then by the regularity properties of $\s$, 
for which see e.g., \cite[chapters 4 and 5]{CS}, or \cite{FK}, we have that 
$\Sz^2_\epsilon (w, z)$ is in $C^\infty (\overline{\Omega})$, as a function of $z$, and is continuous for $w \in \bndry D$. Moreover $\overline{\partial}_z \left( \Sz^2_\epsilon (w, z ) \right) 
= -\overline\dee_z\Omega_0(\eta^{\, \epsilon}) (w, z)= 0$, for $z\in D$ (recall that $D\subset \mathcal P_w$)
so 
$$\overline{\partial}_z \left( \Omega_0(\eta^{\epsilon}) + \Sz^2_\epsilon) \right)\!\!(w, z) = 0
\quad \mbox{for} \ z \in D\quad \mbox{and}\ \  w\in \bndry D.$$

\noindent We similarly have that 
$\B^2_\epsilon (w, z)$ is in $C^\infty (\overline{\Omega})$, as a function of $z$, and is continuous for $w \in \overline{D}$ and, furthermore
$$\overline{\partial}_z \left( (\deebar_w\widetilde\eta^{\,\epsilon})^n + \\B^2_\epsilon) \right)\!\!(w, z) = 0
\quad \mbox{for} \ z \in D\quad \mbox{and}\ \  w\in \overline D.$$
\bigskip

\section{reproducing formulas: globally  holomorphic kernels} \label{S:global-hol}
  At last, in this section we complete the construction of a number of integral operators
  that satisfy all three of the fundamental conditions (a) - (c) that were presented in Section \ref{S:CI-n}. The crucial step in all these constructions  is to produce integral kernels that are globally holomorphic in $D$ as functions of $z$. For strongly pseudo-convex domains, this goal is achieved by adding to each of the (locally holomorphic) Cauchy-Fantappi\'e forms that were produced in Section \ref{S:local-holom} the ad-hoc ``correction'' term that was constructed in Section \ref{S:correction}; the resulting two families of operators are denoted $\{\mathbf \Sz_\epsilon\}_\epsilon$ (acting on $C(\bndry D)$) and $\{\mathbf\B_\epsilon\}_\epsilon$ (acting on $L^1(D)$).
   In the case of strongly $\mathbb C$-linearly convex domains of class $C^2$, there is no need for
    ``correction'': a natural, globally holomorphic Cauchy-Fantappi\'e form is readily available that
    gives rise to an operator acting on $C(\bndry D)$ (even on $L^1(\bndry D)$), called the Cauchy-Leray Integral 
    $\mathbf\Sz_{L}$ and, in the more restrictive setting of strongly convex domains, also to an operator $\mathbf\B_L$ that acts on $L^1(D)$.  (As we shall see in Section \ref{S:main}, in the special case when the domain is the unit ball, 
          the Cauchy-Leray integral $\mathbf\Sz_{L}$ agrees with the Cauchy-Szeg\"o projection $\mathbf S$, while the operator $\mathbf\B_L$ agrees with the Bergman projection $\mathbf\B$.)
All the operators that are produced in this section satisfy, by their very construction, conditions (a) and (c) 
    in Section \ref{S:CI-n}, and  we show in propositions 
    \ref{P:repr-hol-eps-bndry-corr} through \ref{P:repr-CL-interior} that 
    they also satisfy condition (b)
     (the reproducing property for holomorphic functions). 
\subsection{Strongly pseudo-convex domains.} For  $\eta_\epsilon$ is as in proposition \ref{P:repr-hol-eps-bndry} we now write
\begin{equation*}
\Sz^1_\epsilon (w, z) = \Omega_0(\eta^{\epsilon})(w, z)
\end{equation*}
and let
\begin{equation*}
\Sz_\epsilon (w, z) = j^*\!\!\left(\Sz^1_\epsilon (w, z) + \Sz^2_\epsilon (w, z)\right)
\end{equation*}
and we define the operator
\begin{equation}\label{D:HR-integral}
\mathbf \Sz_\epsilon f(z) =
\int\limits_{w\in\bndry D}\!\!\! \!f(w)\,\Sz_\epsilon (w, z),\quad z\in D,\ \ f\in C(\bndry D).
\end{equation}
\begin{Prop}\label{P:repr-hol-eps-bndry-corr}
 Let $D$
be a bounded strongly pseudo-convex domain.
Then, for any $0<\epsilon<\epsilon_0$ 
we have
\begin{equation*}\label{E:repr-hol-eps-bndry-corr}
f(z) = \mathbf \Sz_\epsilon f(z),
\quad \mbox{for any}\ 
f\in \vartheta (D)\cap C(\overline{D}),\ 
z\in D.
\end{equation*}
\end{Prop}
\begin{proof}
By proposition \ref{P:repr-hol-eps-bndry}, for any $f\in \vartheta (D)\cap C(\overline{D})$ we have
$$
\int\limits_{w\in \bndry D}\!\!\!\!\!
f(w)\, \Sz_\epsilon(w, z) = 
f(z) +\!\!\! 
\int\limits_{w\in \bndry D}\!\!\!\!\!
f(w)\, j^*\Sz^2_\epsilon(w, z)\ \mbox{for any}\ 
z\in D,
$$
and so it suffices to show that
\begin{equation*}
\int\limits_{w\in \bndry D}\!\!\!\!\!
f(w)\, j^*\!\Sz^2_\epsilon(w, z)\, =\, 0\  \mbox{for any}\ 
z\in D.
\end{equation*}
By Fubini's theorem and the definition of $\Sz^2_\epsilon$, see \eqref{E:def-correction-szego},  we have
\begin{equation*}
\int\limits_{w\in \bndry D}\!\!\!\!\!
f(w)\, j^*\Sz^2_\epsilon(w, z) \,=\,
\mathcal S_z\!\!\left(\ \int\limits_{w\in \bndry D}\!\!\!\!\!
f(w)\, j^*H(w, \cdot)\right)
\end{equation*}
where $H(w,\cdot)$ is as in \eqref{E:def-d-bar-data-bndry}. Since the solution operator $\mathcal S_z$ is realized as a combinations of integrals over $\Omega$ and $\bndry \Omega$, the desired conclusion will be a consequence of the following claim:
\begin{equation*}\label{E:conclusion-bndry-corr}
\int\limits_{w\in \bndry D}\!\!\!\!\!
f(w)\, j^*H(w, \lam)\,=\,0\quad\mbox{for any}\ \ \lam\in \overline\Omega,
\end{equation*}
and since $\overline\Omega\subset\mathcal P_w$ for any $w\in\bndry D$, proving the latter
amounts to showing that
\begin{equation}\label{E:conclusion-bndry-corr}
\int\limits_{w\in M_\lam}\!\!\!\!\!
f(w)\, j^*\deebar_\lam\Omega_0(\eta^{\epsilon})(w, \lam)\,=\,0\quad\mbox{for any}\ \ \lam\in \overline\Omega,
\end{equation}
where we have set
\begin{equation}\label{D:M}
M_\lam=\{w\in\bndry D\ |\ |w-\lam|\geq\epsilon_0/2\},
\end{equation}
see \eqref{E:def-d-bar-data-bndry} and Figure \ref{fig:Figure2} below. To this end, we fix $\lam\in\overline\Omega$ arbitrarily; we claim that there is a sequence of forms $(\eta^\epsilon_\ell(\cdot, \lam))_\ell$ with the following properties:
\begin{itemize}
\item[{\em a.}\ ] $\eta^\epsilon_\ell(\cdot, \lam)$
 is generating at $\lam$ \ relative to $D$;
\item[{\em b.}\ ] $\eta^\epsilon_\ell(\cdot, \lam)$ 
has coefficients in $C^2(U_\lam)$ with $ U_\lam$
  as in definition \ref{D:gen-form};
 \item[{\em c.}\ ] as $\ell\to\infty$,  we have 
 that 
 $$j^*\Omega_0(\eta^\epsilon_\ell)(\cdot, \lam)\to j^*\Omega_0(\eta^\epsilon)(\cdot, \lam)\quad
 \mbox{uniformly on}\ \  \bndry D;
 $$
 \item[{\em d.}\ ] the coefficients of $\eta^\epsilon_\ell(w, \lam)$ are holomorphic in 
 $ \lam\in\mathbb B_{\epsilon_0/2}(w)$ for any $w\in \bndry D$.
\end{itemize}
Note that \eqref{E:conclusion-bndry-corr} will follow from item {\em c.} above if we can prove that
\begin{equation}\label{E:conclusion-bndry-corr-ell}
\int\limits_{w\in M_\lam}
\!\!\!\!\!
f(w)\, j^*\deebar_\lam\Omega_0(\eta^{\epsilon}_\ell)(w, \lam)\,=\,0\quad
\mbox{for any}\ \ \ell.
\end{equation}
We postpone the construction of $\eta^\epsilon_\ell(\cdot, \lam)$ to later below, and instead proceed to proving \eqref{E:conclusion-bndry-corr-ell} assuming the existence of the 
$\{\eta^\epsilon_\ell(\cdot, \lam)\}_\ell$.
 On account
 of items {\em a.} and {\em b.} above along with basic property \ref{BP-4} as stated in
 \eqref{E:BP-4-aux}, proving
  \eqref{E:conclusion-bndry-corr-ell} is equivalent to showing that
\begin{equation*}\label{E:conclusion-bndry-corr-ell-aux}
\int\limits_{w\in M_\lam}
\!\!\!\!\!
f(w)\, j^*\deebar_w\Omega_1(\eta^{\epsilon}_\ell)(w, \lam)\,=\,0\quad
\mbox{for any}\ \ \ell.
\end{equation*}
 To this end, we first consider the case when $f\in\vartheta (D)\cap C^1(\overline D)$, as in this case we have that
 $$
 f(w) j^*\deebar_w\Omega_1(\eta^{\epsilon}_\ell)(w, \lam) =
 j^*\deebar_w\left(f\,\Omega_1(\eta^{\epsilon}_\ell)\right)\!(w, \lam)
= j^*\!d_w\left(f\,\Omega_1(\eta^{\epsilon}_\ell)\right)\!(w, \lam)
$$
(where in the last identity we have used the fact that $\deebar_w\Omega_1=d_w\Omega_1$ because $\Omega_1(\eta^{\epsilon}_\ell)$ is of type
$(n, n-2)$ in $w$). But the latter equals
$$
{\tt d}_wj^*\!\left(f\,\Omega_1(\eta^{\epsilon}_\ell)\right)\!(w, \lam)
 $$
 where ${\tt d}_w$ denotes the exterior derivative operator for $M_\lam$ viewed as a real manifold of dimension $2n-1$. Applying Stokes' theorem on $M_\lam$ to the form 
 $\alpha(w):= j^*\!\left(f\,\Omega_1(\eta^{\epsilon}_\ell)\right)\!(w, \lam)\in
  C^1_{n, n-2}(M_\lam)$ we obtain
 $$
 \int\limits_{w\in M_\lam}
\!\!\!\!\!
f(w)\, j^*\deebar_w\Omega_1(\eta^{\epsilon}_\ell)(w, \lam)\,=\,
\int\limits_{w\in \bndry M_\lam}
\!\!\!\!\!
f(w)\, j^*\Omega_1(\eta^{\epsilon}_\ell)(w, \lam)\,
 $$
 but
 $$
 j^*\Omega_1(\eta^{\epsilon}_\ell)(w, \lam) =0\quad \mbox{for any}\ 
 w\in\bndry M_\lam=\bndry D\cap\{|w-\lam|=\epsilon_0/2\}
 $$
 because the coefficients of $\eta^\epsilon_\ell(w,\lam)$ are holomorphic in 
 $\lam\in \mathbb B_{\epsilon/2}(w)$ for any $\bndry D$, see
  \eqref{D:CF-1} and item {\em d.} above.
  This concludes the proof of proposition \ref{P:repr-hol-eps-bndry-corr} in the case when
  $f\in \vartheta (D)\cap C^1(\overline D)$. 
  
  To prove the proposition in the case when $f\in\vartheta(D)\cap C^0(\overline D)$,
   we fix $z\in D$ and
choose $\delta = \delta (z)>0$ such that 
$$
z\in D_{-\delta} = \{\rho <-\delta \}\ \ \mathrm{for\ any}\ \ \delta\leq \delta(z). 
$$
Then we have that
$$
f\in\vartheta (D_{-\delta})\cap C^1(\overline D_{-\delta})\quad \mathrm{for\ any}\ \ \delta\leq \delta(z)
$$
and so by the previous argument we have
\begin{equation}\label{E:I-delta}
\int\limits_{w\in\bndry D_{-\delta}}\!\!\!\!\!\!\!\! f(w)j_{-\delta}^*\Sz^2_{\epsilon}(w, z) = 0\quad 
\mathrm{for\ any}\ \ \delta\leq \delta(z),
\end{equation}
where $j^*_{-\delta}$ denotes the pullback under the inclusion:
$\bndry D_{-\delta}\hookrightarrow \mathbb C^n$. For $\delta$ sufficiently small there is a natural
one-to-one and onto projection along the inner normal direction:
$$
\Lambda_\delta: \bndry D\to \bndry D_{-\delta},
$$
and because $D$ is of class $C^2$ one can show that this projection tends 
  in
the $C^1$-norm to the identity $\mathbf{1}_{\bndry D}$, that is we have that
$$
\|\mathbf{1}_{\bndry D} -\Lambda_\delta\|_{C^1(\bndry D)}\to 0\quad \mbox{as}\ \ \delta\to 0.
$$
Using this projection one may then express the integral on $\bndry D_{-\delta}$ in identity \eqref{E:I-delta} as an integral
on $\bndry D$ for an integrand that now also depends on $\Lambda_{\delta}$ and its Jacobian,
and it follows from the above considerations that
$$
\int\limits_{w\in\bndry D_{-\delta}}\!\!\!\!\!\!\!\! f(w)j_{-\delta}^*\Sz^2_{\epsilon}(w, z) \ \ 
\to 
\int\limits_{w\in\bndry D}\!\!\!\!\! f(w)j^*\Sz^2_{\epsilon}(w, z)\quad \mbox{as}\ \ \delta\to 0.
$$



We are left to construct, for each fixed $\lam\in\overline\Omega$,  
the sequence
$\{\eta^{\,\epsilon}_\ell(\cdot, \lam)\}_\ell$
 that was invoked earlier on.
To this end, set 
\begin{equation*}
U:=D\cup \bigcup\limits_{z\in D} U_z
\end{equation*}
where $U_z$ is the open neighborhood of $\bndry D$ that was determined in lemma
\ref{L:gen-form-pscvx-eps}. 
Consider a sequence of real-valued functions $\{\rho_\ell\}_\ell\subset C^3(\mathbb C^n)$ such that
\begin{equation*}\label{E:C-2-conv}
\|\rho_\ell-\rho\|_{C^1(U)}\to 0\quad \mbox{as\ }\ \ell\to\infty,
\end{equation*}
and, for $\lam\in\overline\Omega$ fixed arbitrarily, set
 \begin{equation*}\label{E:smooth-ell}
\Delta_{j,\,\ell}^\epsilon(w, \lam)\ :=\ \frac{\dee\rho_\ell}{\dee\zeta_j}(w) -\frac{1}{2}
\sum\limits_{k=1}^n\tau_{j, k}^\epsilon(w)(w_k-\lam_k),\quad j=1,\ldots, n;
\end{equation*}

  \begin{equation*}
 \Delta^{\!\epsilon}_\ell(w, \lam) := \sum\limits_{j=1}^n \Delta_{j,\,\ell}^\epsilon(w, \lam)\,(w_j-\lam_j);
\end{equation*}
 
 and, for $\chi_1$ as in \eqref{D:chi-0}:

\begin{equation*}
g^{\epsilon}_\ell(w, \lam):= \chi_1(w, \lam)\Delta^{\!\epsilon}_\ell (w,\lam) +(1-\chi_1(w, \lam))|w-\lam|^2;
\end{equation*}

\begin{equation*}
\eta_{j,\, \ell}^{\epsilon}(w, \lam)\ := \ 
\frac{1}{g^\epsilon_\ell(w, \lam)}\bigg(\chi_1(w, \lam)\Delta_{j,\,\ell}^{\!\epsilon}(w, \lam) + (1-\chi_1(w, \lam))(\overline{w}_j-\overline{\lam}_j)\!\bigg)
\end{equation*}
and, finally
\begin{equation*}
\eta^\epsilon_\ell (w, \lam) :=\sum\limits_{j=1}^n\eta_{j,\,\ell}^\epsilon(w, \lam)\,dw_j\, ,
\end{equation*}
We leave it as an exercise for the reader to verify that 
$\{\eta^{\,\epsilon}_\ell(\cdot, \lam)\}_\ell$ has the desired properties.
\end{proof}

 \begin{figure}[h!]
 \centering 
 \includegraphics[width=0.5 \textwidth]
{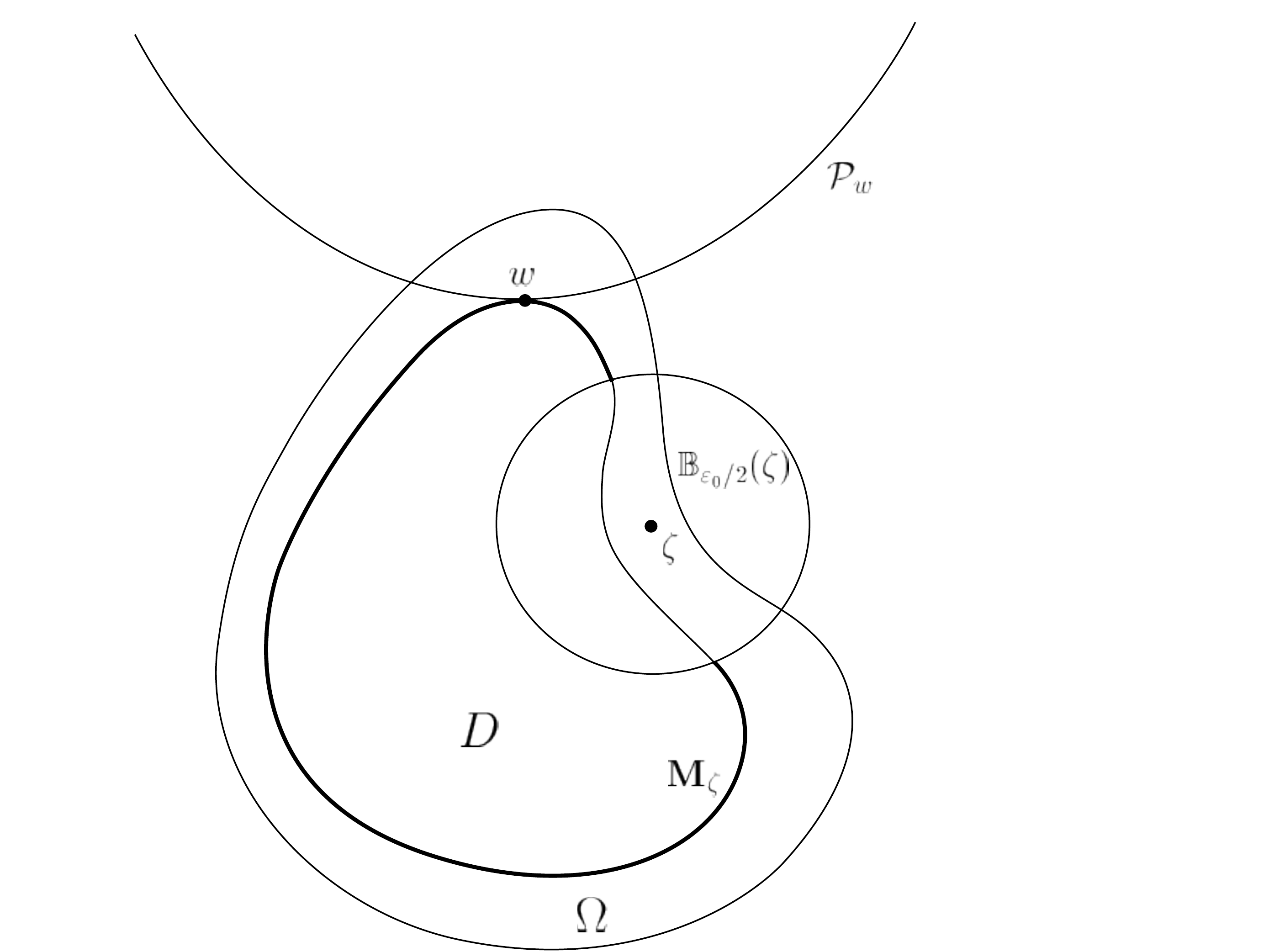}
\caption{The manifold $M_\zeta$ in the proof of proposition \ref{P:repr-hol-eps-bndry-corr}.}
\label{fig:Figure2}
\end{figure}

Next,  for $\widetilde\eta^{\,\epsilon}$ is as in proposition \ref{P:repr-hol-eps-interior}, we write

\begin{equation*}
 \B^1_\epsilon(w, z) = \frac{1}{(2\pi i)^n}(\deebar_w\widetilde\eta^{\,\epsilon})^n
\end{equation*}
and
\begin{equation*}
\B_\epsilon (w, z) := \left( \B^1_\epsilon + \B^2_\epsilon\right)\!(w, z),\quad w\in\overline D,\quad z\in \overline{\Omega},
\end{equation*}

and we define the operator
\begin{equation}\label{D:HRL-op}
\mathbf B_\epsilon f(z) =\int\limits_{w\in D}\!\!\! f(w)\, \B_\epsilon (w, z), \quad z\in D, \quad f\in L^1(D).
\end{equation}
\begin{Prop}\label{P:repr-hol-eps-interior-corr}  Let $D$
be a bounded strongly pseudo-convex domain.
Then, for any $0<\epsilon<\epsilon_0$ 
we have
\begin{equation*}\label{E:repr-hol-eps-interior-corr}
f(z) = \mathbf B_\epsilon f(z),
\quad \mbox{for any}\ 
f\in \vartheta (D)\cap L^1(D),\ 
z\in D.
\end{equation*}
\end{Prop}
\begin{proof}
By proposition \ref{P:repr-hol-eps-interior}, for any $f\in \vartheta (D)\cap L^1(D)$ we have
$$
\int\limits_{w\in D}\!\!\!\!
f(w)\B_\epsilon(w, z) = 
f(z) +\!\!\! 
\int\limits_{w\in D}\!\!\!\!
f(w)\B^2_\epsilon(w, z)\ \mbox{for any}\ 
z\in D,
$$
and so it suffices to show that
\begin{equation*}
\int\limits_{w\in D}\!\!\!\!
f(w)\B^2_\epsilon(w, z)\, =\, 0\  \mbox{for any}\ 
z\in D.
\end{equation*}
For the proof of this assertion we refer to \cite[proposition 3.2]{LS-2}.
\end{proof}

\bigskip

\subsection{}{\bf Strictly $\mathbb C$-linearly convex domains: the Cauchy-Leray integral.}\label{SS:spc}
 Let $D$ be a bounded, strictly $\mathbb C$-linearly convex domain.
 We claim that if $\rho$ is (any) defining function for such a domain,
 and if $U$ is an open neighborhood of $\bndry D$ such that $\nabla\rho (w)\neq 0$ for any $w\in U$, then 
\begin{equation}\label{E:gen-ball}
\eta(w, z) := \frac{\dee\rho(w)}{\langle\dee\rho (w), w-z\rangle}
\end{equation}
is a generating form for $D$;
indeed, by 
lemma \ref{L:strictly-C-lin-cvx} for any $z\in D$ 
there is an open set $U_z\subset \mathbb C^n\setminus\{z\}$
 such
 that $\langle\dee\rho (w), w-z\rangle\neq 0$ for any $w\in U_z$ and 
 $\bndry D\subset U_z$; thus the coefficients of 
  $\eta (\cdot, z)$ are in $C(U_z)$ and \eqref{D:generating-2} holds. It is clear 
  from \eqref{E:gen-ball} that
 $\langle\eta (w, z), w-z\rangle = 1$ for any $w\in U_z$, so \eqref{D:generating-1} holds for any $z\in D$, as well.
It follows that proposition \ref{P:repr-hol-bndry} applies to any strictly $\mathbb C$-linearly convex domain $D$  with
 $\eta$ chosen as above under the further assumption that $D$ be of class $C^2$ (which is required to ensure that the coefficients of $\eta (\cdot, z)$ are in $C^1(U_z)$). 
  The form
\begin{equation}\label{E:szego}
\Sz_L(w, z)\! =\! j^*\Omega_0\!\left(\frac{\dee\rho (w)}{\langle\dee\rho (w), w-z\rangle}\right) 
\! =\! 
j^*\!\!\left(\frac{\dee\rho (w)\wedge (\deebar\dee\rho)^{n-1}(w)}{(2\pi i\langle\dee\rho (w), w-z\rangle)^n}\right)
\end{equation}
is called the {\em Cauchy-Leray kernel for $D$}.
It is clear that
   the coefficients of 
the Cauchy-Leray kernel
 are \emph{globally} holomorphic with respect to $z\in D$:
 indeed the denominator $j^*\langle \dee\rho(w), w-z\rangle^n$ is polynomial in the variable $z$, and by the strict $\mathbb C$-linear convexity of $D$ we have that 
 $j^*\langle \dee\rho(w), w-z\rangle^n\neq 0$ for any $z\in D$ and for any $w\in\bndry D$,
   see \eqref{E:strictly-C-lin-cvx}. The resulting integral operator:
  \begin{equation}\label{D:CL-smooth}
  \mathbf \Sz_L f(z) =  
    \int\limits_{w\in\bndry D}\!\!\!\!\! f(w)\, C_L(w, z)
  \ \ z\in D,
  \end{equation}
  is called the {\em Cauchy-Leray Integral}.
Under the further assumption that $D$ be strictly convex (as opposed to strictly 
{\em $\mathbb C$-linearly} convex), for each fixed $z\in D$ one may
extend
 $\eta (\cdot, z)$ holomorphically to the interior of $D$ as follows
\begin{equation}\label{D:h}
   \widetilde\eta (\cdot, z)\!:=\!
    \left(\!\frac{\langle\dee\rho(\cdot), \cdot-z\rangle}{\langle\dee\rho(\cdot), \cdot-z\rangle-\rho(\cdot)}\!\right)\!
  \eta(\cdot, z) =   
  \frac{\dee\rho(\cdot)}{\langle\dee\rho(\cdot), \cdot-z\rangle-\rho(\cdot)}
\end{equation} 
  The following lemma shows that if  $D$ is sufficiently smooth (again of class $C^2$)
   then $\widetilde\eta$  satisfies the hypotheses of proposition
    \ref{P:repr-hol-interior}, and so in particular the operator
    \begin{equation*}\label{D:solid-CL}
    \mathbf \B_L f(z) =\int\limits_{w\in D}\!\!\! f(w)\,\B_L(w, z)
       \end{equation*}
  with 
  \begin{equation}\label{D:bergman-ker}
  \B_L(w, z)=\frac{1}{(2\pi i)^n}(\deebar_w\widetilde\eta)^n(w, z)
  \end{equation}
 and  $\widetilde \eta$ given by \eqref{D:h},  reproduces holomorphic functions\footnote{Note that $\widetilde\eta$ does not satisfy the stronger condition   \eqref{E:extension-stronger} that was discussed earlier.}. 
    
    \begin{Lem} If $D=\{\rho<0\}\subset\mathbb C^n$
     is strictly convex and of class $C^2$, then for each fixed $z\in D$ we have that 
     $\widetilde\eta (\cdot, z)$ given by \eqref{D:h} has coefficients in 
     $C^1(\overline{D})$
        and satisfies the hypotheses of proposition \ref{P:repr-hol-interior}.
       \end{Lem}
   \begin{proof}
   In order to prove the first assertion it suffices to show that
   \begin{equation}\label{E:Lem}
   \mbox{Re}\left(\langle\dee\rho(w), w-z\rangle\right)
   -\rho (w) > 0\quad \mbox{for any }\ w\in\overline{D},\ \ z\in D.
   \end{equation}
   Indeed, one first observes that if $D$ is strictly convex and sufficiently smooth
   then
   $$
       \mbox{Re}\langle\dee\rho(w), w-z\rangle>0\quad \mbox{for any}\  \ w\in \overline{D}\setminus\{z\}
   $$
 (see \cite{Ho}  for the proof of this fact) so that $\mbox{Re}\langle\dee\rho(w), w-z\rangle$ is non-negative in $\overline{D}$ and it vanishes only at $w= z$. On the other other hand 
 the term $-\rho(w)$ is non-negative for any $w\in \overline{D}$, and if $w=z\in D$ then $-\rho(w)=-\rho(z)>0$. This proves \eqref{E:Lem} 
and it follows that the coefficients of 
$\widetilde\eta(\cdot, z)$ are in $C^1(\overline{D})$. By basic property \ref{BP-1} we have
 \begin{equation*}
   \Omega_0(\widetilde\eta) (\cdot, z) =
   \left(\!\frac{\langle\dee\rho(\cdot), \cdot-z\rangle}{\langle\dee\rho(\cdot), \cdot-z\rangle-\rho(\cdot)}\!\right)^{\!n}\!\Omega_0(\eta)(\cdot, z);
\end{equation*}
 it is now immediate to
 verify that
 $
 j^*\Omega_0(\widetilde\eta) (\cdot, z) = 
  j^*\Omega_0(\eta) (\cdot, z),
  $
 so that $\widetilde\eta$ satisfies \eqref{E:extension}, as desired.
   \end{proof} 
We summarize these results in the following two propositions:
\begin{Prop}\label{P:repr-CL-bndry} Suppose that $D$ is a bounded, strictly $\mathbb C$-linearly convex domain of class $C^2$. Then, with same notations as above we have
\begin{equation*}
f(z) = \mathbf\Sz_{L}f(z),\quad z\in D, \quad f\in \vartheta(D)\cap C(\overline D).
\end{equation*}
\end{Prop}
\begin{Prop}\label{P:repr-CL-interior} Suppose that $D$ is a bounded, strictly convex domain of class $C^2$. Then, with same notations as above we have that
\begin{equation*}
f(z) = \mathbf\B_{L}f(z),\quad z\in D, \quad f\in \vartheta(D)\cap L^1(D).
\end{equation*}
\end{Prop}

\bigskip

\section{$L^p$ estimates}\label{S:main}

In this section we discuss $L^p$-regularity of the Cauchy-Leray integral and of the Cauchy-Szeg\"o and Bergman projections for the domains under consideration. Detailed proofs of the results concerning the Bergman projection, theorem \ref{T:Bergman-bdd} and corollary \ref{C:abs-Bergman} below,  can be found in \cite{LS-2}. 
The statements concerning the Cauchy-Leray integral and the Cauchy-Szeg\"o projection
 (theorems \ref{T:CL-bdd-smooth} and \ref{T:Szego-bdd} below,
 and
 theorem \ref{T:CL-bdd-C11} in the next section)  are the subject of a series of forthcoming papers; here we will limit ourselves to presenting an outline of the main 
points of interest in their proofs.

We begin by recalling the defining properties of the Bergman and Cauchy-Szeg\"o projections and of
their corresponding function spaces.\\

\subsection{The Bergman Projection}
Let
$
D\subset\mathbb C^n
$
be a bounded connected open set.

\begin{Def}
For any $1\leq q<\infty$ the {\em Bergman space $\vartheta L^q(D)$} is
$$
\vartheta L^q(D) =\vartheta (D)\cap L^q(D, dV).
$$
\end{Def}
The following inequality
$$
\sup\limits_{z\in\mathcal K}|F(z)| \leq C(\mathcal K)\|F\|_{L^p(D, dV)}
$$
 which is valid for any compact subset $\mathcal K\subset D$
  and for any holomorphic function 
 $F\in\vartheta(D)$, shows that the Bergman space 
 is a closed subspace of $L^q(D, dV)$. This inequality also shows that
  the point evaluation:
$$
ev_z(f):= f(z),\quad z\in D
$$
is a bounded linear functional on the Bergman space 
(take $\mathcal K := \{z\}$). 
In the special case $q=2$, classical arguments
from the theory of Hilbert spaces grant the existence of an orthogonal projection,
called the {\em Bergman projection for $D$}
$$
\Bop: L^2(D)\to \vartheta L^2(D)
$$
that enjoys the following properties
\begin{equation*}\label{E:bergman-ker-5}
\Bop f(z) = f(z),\quad f\in \vartheta L^2(D),\ \ z\in D
\end{equation*}
\begin{equation*}\label{E:bergman-1}
\Bop^*=\Bop
\end{equation*}
\begin{equation*}\label{E:bergman-2}
 \|\Bop f\|_{L^2(D, dV)}\leq \|f\|_{L^2(D,\, dV)},\quad f\in L^2(D,\, dV)
 \end{equation*}
\begin{equation*}\label{E:bergman-ker-1}
 \Bop f(z) = \int\limits_{w\in D}\!\!\! f(w)\,\mathcal B(w, z)\,dV(w),\quad z\in D,\ \ f\in L^2(D, dV)
 \end{equation*}
 where $dV$ denotes Lebesgue measure for $\mathbb C^n$.
%
The function $\mathcal B(w, z)$ is holomorphic with respect to $z\in D$; it
is called the 
{\em Bergman kernel function}. The Bergman kernel function depends on the domain and is known explicitly only for very special domains, such as the unit ball, see e.g. \cite{Ru}:
\begin{equation}\label{E:bergman-ker-disc}
\mathcal B(w, z) = \frac{n!}{\pi^n(1-[z, w])^{n+1}}\,,\quad (w, z)\in\mathbb B_1(0) \times\mathbb B_1(0)
\end{equation}
here 
$\(z, w\) := \sum\limits_{j=1}^nz_j\cdot\overline{w}_j$ is the hermitian product  for $\mathbb C^n$.\\

\subsection{The Cauchy-Szeg\"o projection.} 
Let
$
D\subset\mathbb C^n
$
be a bounded connected open set with 
sufficiently smooth
 boundary. For such a domain, various notions of Hardy spaces of holomorphic functions can be obtained by considering (suitably interpreted) boundary
 values of functions  that are holomorphic in $D$ and whose restriction to the boundary of $D$ has some integrability, see \cite{S}. 
 While a number of such definitions can be given, here we 
 adopt the following

\begin{Def}
For any $1\leq q<\infty$ the {\em Hardy Space $H^q(\bndry D, d\sigma)$} is the closure in $L^q(\bndry D, d\sigma)$ of the restriction to the boundary of the functions holomorphic in a neighborhood of $\overline{D}$.
In the special case when $q=2$ the orthogonal projection
$$
\mathbf S: L^2(\bndry D, d\sigma)\to H^2(\bndry D, d\sigma)
$$
is called the {\em The Cauchy-Szeg\"o Projection
  for $D$}.
\end{Def}
The Cauchy-Szeg\"o projection has the following basic properties:

\begin{equation*}\label{E:szego-1}
\mathbf S^*=\mathbf S
\end{equation*}
\begin{equation*}\label{E:szego-2}
 \|\mathbf S f\|_{L^2(\bndry D, \, d\sigma)}\leq \|f\|_{L^2(\bndry D, \, d\sigma)}\, ,\quad f\in L^2(\bndry D, d\sigma)
 \end{equation*}
\begin{equation*}\label{E:szego-ker-1}
 \mathbf S f(z) = \int\limits_{w\in \bndry D}\!\!\!\! \mathcal S(w, z)\,f(w)d\sigma(w),\quad z\in \bndry D.
 \end{equation*}
The function $\mathcal S(w, z)$, initially defined for $z\in\bndry D$, extends holomorphically to
$z\in D$; it is called the 
{\em Cauchy-Szeg\"o kernel function}. Like the Bergman kernel function, the Cauchy-Szeg\"o kernel function depends on the domain $D$; for the unit ball we have \cite{Ru}
\begin{equation}\label{E:szego-ker-disc}
\mathcal S(w, z) = \,\frac{(n-1)!}{2\pi^n(1-\(z, w\))^n},\quad 
(w, z)\in \bndry\mathbb B_1(0)\times \bndry 
\mathbb B_1(0)\, .\\
\end{equation}

\subsection{$L^p$-estimates}
We may now state our main results.
\begin{Thm}\label{T:CL-bdd-smooth}
 Suppose $D$ is a bounded domain of class $C^2$ which is strongly $\mathbb{C}$-linearly convex. Then the Cauchy-Leray integral \eqref{D:CL-smooth},  initially defined for $f \in C^1 (\bndry D )$, extends to a bounded operator on $L^p (\bndry D, d\sigma)$, $1 < p < \infty$. 
\end{Thm}
It is only the weaker notion of {\em strict}\, $\mathbb{C}$-linear convexity that is needed to define the Cauchy-Leray integral, but to prove the $L^p$ results one needs to assume
{\em strong} $\mathbb{C}$-linear convexity.
\begin{Thm}\label{T:Szego-bdd}
   Under the assumption that the bounded domain $D$ has a $C^2$ boundary and is strongly pseudo-convex, one can assert that $\mathbf S$ extends to a bounded mapping on $L^p (\bndry D , d \sigma )$, when $1 < p < \infty$. 
\end{Thm}
\begin{Thm}\label{T:Bergman-bdd}
   Under the same assumptions on $D$ it follows that the operator $\mathbf B$ extends to a bounded operator on $L^p (D , dV)$ for $1 < p < \infty$.
\end{Thm}
\indent The following additional results also hold.
\begin{Cor}\label{C:weighted-bdd}
   The result of theorem \ref{T:Szego-bdd} extends to the case when the projection $\mathbf S$ is replaced by the corresponding orthogonal projection $\mathbf S_\omega$, with respect to the Hilbert space $L^2 (\bndry D , \omega d \sigma )$ where $\omega$ is any continuous strictly positive function on $\bndry D$.
\end{Cor}
\indent 
A similar variant of theorem
\ref{T:Bergman-bdd} holds
for $\mathbf B_\omega$, the orthogonal projection on the sub-space of $L^2 (D , \omega\, dV)$. Here $\omega$ is any strictly positive continuous function on $\overline{D}$. 
\begin{Cor}\label{C:abs-Bergman}
   One also has the $L^p$ boundedness of the operator $| B |$, whose kernel is $| \mathcal B(z,w) |dV(w)$, where $\mathcal B(z,w)$ is the Bergman kernel function.\\
\end{Cor}

\subsection{Outline of the proofs} We begin by making the following remarks to clarify the background of these results.
\smallskip

\begin{enumerate}

\item The proofs of theorems \ref{T:Szego-bdd} and \ref{T:Bergman-bdd} make use of the whole 
{\em family} of operators $\displaystyle{\{\mathbf \Sz_\epsilon\}_\epsilon}$, $0<\epsilon<\epsilon_0$: in order to obtain $L^p$ estimates for $p$ in the full range $(1, \infty)$
one needs the flexbility to choose $\epsilon =\epsilon (p)$ sufficiently small. 
(A single choice, as in \cite{Ra}, of
$\mathbf \Sz_\epsilon$ for a fixed $\epsilon$, will not do.)
\medskip

   \item There is no simple and direct relation between $\mathbf S$ and $\mathbf S_\omega$, nor between $\mathbf B$ and $\mathbf B_\omega$. Thus the results for general $\omega$ are not immediate consequences of the results for $\omega \equiv 1$.
   \medskip
   
\item When $\bndry D$ and $\omega$ are smooth (i.e. $C^k$ for sufficiently high $k$), the above results have been known for a long time (see e.g., the remarks that were made in Section \ref{S:global-hol} concerning the case when $D$ is the unit ball).
 Moreover when $\bndry D$ and $\omega$ are smooth (and $\bndry D$ is strongly pseudo-convex), there are analogous asymptotic formulas for the kernels in question due to \cite{F}, which allow a proof of theorems \ref{T:Szego-bdd} and \ref{T:Bergman-bdd} in these cases. See also \cite{PS}.
 \medskip
 
\item Another approach to theorem \ref{T:Bergman-bdd} in the case of smooth strongly pseudo-convex domains is via the $\overline{\partial}$-Neumann problem \cite{CS} and \cite{FK}, but we shall not say anything more about this here.
\end{enumerate}
\bigskip

A further point of interest is to work with the \textquotedblleft Levi-Leray\textquotedblright \ measure $d \mu_\rho$ for the boundary of $D$, which we define as follows.
We take the linear functional
\begin{equation}\label{D:levi-leray}
   \ell (f) = \frac{1}{(2 \pi i)^n} \int_{\bndry D}\limits f(w)\, j^*\!\!\left(\partial \rho  \wedge (\overline{\partial} \partial \rho )^{n-1}\right)
\end{equation}
and write $\displaystyle \ell {f} = \int_{\bndry D}\limits\!\! f d\mu_\rho$. We then have
that $d\mu_\rho (w) = \D(w)d\sigma (w)$ where
 $\D(w) = c | \nabla\rho (w)|\det L_w(\rho)$
 via the calculation in \cite{Ra} in the case $\rho$ is of class $C^2$, and we observe that $\D(w) \approx 1$, via \eqref{E:strongly-L-pscvx-III}.\\
 
With this we have that the Cauchy-Leray integral becomes
\begin{equation}\label{Eq:8}
   \mathbf \Sz_L (f)(z) = \int_{\bndry D}\limits \frac{f(w)d\mu_\rho (w)}
   {\langle\dee\rho (w), w-z\rangle^n}
\end{equation} 

 Thus the reason for isolating the measure $d\mu_\rho$ is that 
 the coefficients of the kernel of each of 
 $\mathbf \Sz_L$ and its adjoint (computed with respect to $L^2(\bndry D, d\mu_\rho)$), are 
 $C^1$ functions in both
variables. This would not be the case if we replaced $d\mu_\rho$ by the induced Lebesgue measure $d\sigma$ (and had taken the adjoint of $\mathbf \Sz_L$ with respect to $L^2(\bndry D, d\sigma)$).\\

\indent In studying~(\ref{Eq:8}) we apply the \textquotedblleft T(1)-theorem\textquotedblright \ technique \cite{DJS}, where the underlying geometry is determined by the quasi-metric
 $$|\langle\dee\rho (w), w-z\rangle|^{\frac{1}{2}}$$
 (It is at this juncture that the notion of {\em strong}
  $\mathbb C$-linear convexity, as opposed to {\em strict}  $\,\mathbb C$-linear convexity, is required.)
In this metric, the ball centered at $w$ and reaching to $z$
has $d\mu_\rho$-measure $\approx |\langle\dee\rho (w), w-z\rangle|^n$. \\

The study of \eqref{Eq:8} also requires that we verify cancellation properties in terms of its action on \textquotedblleft bump functions.\textquotedblright \ These matters again differ from the case $n=1$, and in fact there is an unexpected favorable twist: the kernel in \eqref{Eq:8} is an appropriate derivative, as can be surmised
by the observation that on the Heisenberg group one has $\displaystyle ( | z |^2 + it)^{-n} = c^\prime \frac{d}{dt} ( | z |^2 + it )^{-n+1}, \ \text{ if } n > 1$. (However for $n=1$, the corresponding identity involves the logarithm!). Indeed by an integration-by-parts argument 
that is presented in \eqref{Eq:7} below,
 we see that when $n > 1$  and $f$ is of class $C^1$, 
\begin{equation*}
   \mathbf \Sz_L (f)(z) = c \int_{\bndry D}\limits \frac{d\,\!f(w) \wedge j^\ast (\overline{\partial} \partial\rho )^{n-1}}{\langle\dee\rho (w), w-z\rangle^{n-1}} + \mathbf E(f)(z) ,
\end{equation*}
where $$\displaystyle \mathbf E(f)(z) = \int_{\bndry D}\limits\!\! \mathcal E(z,w)f(w)\, d\sigma (w)$$
with $$\mathcal E(z,w) = O( | z-w | \ |\langle\dee\rho (w), w-z\rangle|^{-n} )$$ 

so that the operator $\mathbf E$ is a negligible term.\\

A final point is that the hypotheses of theorem \ref{T:CL-bdd-smooth} are in the nature of best possible. In fact, 
\cite{BaLa} gives examples of Reinhardt domains where the $L^2$ result for the Cauchy-Leray integral fails when a condition near $C^2$ is replaced by $C^{2-\epsilon}$, or \textquotedblleft strong\textquotedblright \ pseudo-convexity is replaced by its \textquotedblleft weak\textquotedblright \ analogue.\\

One more observation concerning the Cauchy-Leray integral is in order.  In the special case when $D$ is the unit ball $\mathbb B_1(0)$, we claim that the operators $\mathbf \Sz_{L}$ and $\mathbf\B_{L}$ agree, respectively, with the Cauchy-Szeg\"o and Bergman projections for $\mathbb B_1(0)$.
Indeed, for such domain the calculations in Section \ref{SS:spc}
 apply with $U_z=\mathbb C^n\setminus\{z\}$ 
and
\begin{equation}\label{E:def-funct-ball}
\rho(w):=|w|^2-1
\end{equation} 
and  by
 the Cauchy-Schwarz inequality we have
$\mbox{Re}\left(\langle\dee\rho (w), w-z\rangle\right)\geq |w|(|w|-|z|)\ \mbox{for any}\ \ w, z\in\mathbb C^n.$ 
Using \eqref{E:def-funct-ball}
and \eqref{E:surface-meas}\footnote{along with the following, easily verified identity: $*\dee\rho(w) =
\dee\rho(w)\wedge (\deebar\dee\rho(w))^{n-1}$.} we find that
$$
\Sz_{L}(w, z)\ =\ 
\frac{(n-1)!}{2\pi^n}\frac{d\sigma(w)}{(1-\(z, w\))^n}\ =\ \mathcal S(w, z)\, d\sigma
$$
which is the Cauchy-Szeg\"o kernel for the ball, see \eqref{E:szego-ker-disc}
Next, we observe that, again for $D=\mathbb B_1(0)$ and with $\rho$ as in 
\eqref{E:def-funct-ball}, we have that
$$
\langle\dee\rho (w), w-z\rangle - \rho(w)\ =\
1-[z, w]\quad \mbox{for any}\ \ w, z\in \mathbb C^n
$$
and from this it follows that \eqref{D:bergman-ker} now reads
$$
\B_L(w, z) \ =\  \frac{n!\, dV(w)}{\pi^n(1-\(z, w\))^{n+1}}\ =\ \mathcal B(w, z)\, dV(w)
$$
which is the Bergman kernel of the ball, see \eqref{E:bergman-ker-disc}.\\

There are three main steps in the proof of theorem \ref{T:Szego-bdd}. \\
\begin{enumerate}
   \item[\em{(i)}\ \ ] Construction of a family of bounded Cauchy Fantappi\'{e}-type integrals 
   $\mathbf \Sz_\epsilon$
   \item[]
\item[\em{(ii)}\ \ ] Estimates for $\mathbf \Sz_\epsilon - \mathbf\Sz_\epsilon^\ast$
\item[]
\item[\em{(iii)}\ \ ] Application of a variant of identity \eqref{Eq:1}\\
\end{enumerate}
{\em Step (i)}. \ The construction of $\mathbf \Sz_\epsilon$ was given in sections 
\ref{S:local-holom} through \ref{S:global-hol}, see \eqref{D:HR-integral}.
One notes that the kernel $\Sz^2_\epsilon(w, z)$ of the correction term that was produced in Section
\ref{S:correction} is \textquotedblleft harmless \textquotedblright \ since it is bounded as $(w, z)$ ranges over $\bndry D\times \overline{D}$. 
Using a methodology similar to the proof of theorem \ref{T:CL-bdd-smooth} one then shows
\begin{equation*}
   \| \mathbf \Sz_\epsilon (f) \|_{L^p} \leq c_{\epsilon , p} \| f \|_{L^p} , \quad 1 < p < \infty .
\end{equation*}
However it is important to point out, that in general the bound $c_{\epsilon , p}$ grows to infinity as $\epsilon \rightarrow 0$, so that the $\mathbf \Sz_\epsilon$ can \underline{not} be genuine approximations of $\mathbf S$. Nevertheless we shall see below that in a sense the $\mathbf \Sz_\epsilon$ gives us critical information about $\mathbf S$.\\

\noindent 
{\em Step (ii).} \ Here the goal is the following splitting:
\begin{Prop}\label{P:split}
   Given $0<\epsilon<\epsilon_0$, 
   we can write
   \begin{equation*}
      \mathbf \Sz_\epsilon - \mathbf \Sz_\epsilon^\ast= \mathbf A_\epsilon + \mathbf R_\epsilon 
   \end{equation*}
   where
   \begin{equation}\label{E:bound-essential}
   \| \mathbf A_\epsilon \|_{L^p \rightarrow L^p} \leq \epsilon c_p \ , \quad 1 < p < \infty
   \end{equation}
   and the operator $\mathbf R_\epsilon$ has a bounded kernel, hence $\mathbf R_\epsilon$ maps $L^1 (\bndry D )$ to $L^\infty (\bndry D )$.
\end{Prop}
\indent We note that in fact the bound of the kernel of $\mathbf R_\epsilon$ may grow to infinity as $\epsilon \rightarrow 0$. \\

To prove proposition \ref{P:split} we first verify an important \textquotedblleft symmetry\textquotedblright \ condition: for each $\epsilon$, there is a $\delta_\epsilon$, so that 
\begin{equation}\label{Eq:10}
  | g^\epsilon (w, z) - \overline{g^{\,\epsilon}} (z, w) |\ \leq\ \epsilon\, c\, | w -z|^2 ,      
  \quad \mbox{if}\quad | w-z | < \delta_\epsilon\, .                                                            \end{equation} 

Here $g_\epsilon (w, z)$ is as in \eqref{D:g-eps}.
With this one proceeds as follows. Suppose $H_\epsilon (z, w)$ is the kernel of the operator $\mathbf \Sz_\epsilon - \mathbf \Sz_\epsilon^\ast$. Then we take $\mathbf A_\epsilon$ and 
$\mathbf R_\epsilon$ to be the operators with kernels respectively $\chi_\delta (w-z) H_\epsilon (w, z)$ and $(1 - \chi_\delta (w-z)) H_\epsilon (w, z)$, where $\chi_\delta (w-z)$
is as in \eqref{D:g-eps}
  and $\delta = \delta_\epsilon$, chosen acccording to 
 \eqref{Eq:10}.\\

\noindent
  {\em Step (iii). }\  We conclude the proof of theorem \ref{T:Szego-bdd} by using an identity similar to \eqref{Eq:1}:
\begin{equation*}\label{Eq:11}
   \mathbf S (\mathbf I - (\mathbf \Sz_\epsilon^\ast - \mathbf \Sz_\epsilon )) = \mathbf \Sz_\epsilon \, 
\end{equation*} 
Hence 
$$\mathbf S(\mathbf I - \mathbf A_\epsilon ) = \mathbf \Sz_\epsilon + \mathbf S
\mathbf R_\epsilon$$
Now for each $p$, take $\epsilon > 0$ so that for the bound $c_p$ as in \eqref{E:bound-essential}
$$\epsilon\, c_p \leq \frac{1}{2}.$$
Then $\mathbf I - \mathbf A_\epsilon$ is invertible and we have
\begin{equation*}
   \mathbf S = \left( \mathbf \Sz_\epsilon + \mathbf  S\mathbf R_\epsilon \right)
    \left( \mathbf I - \mathbf A_\epsilon \right)^{-1}
\end{equation*}
Since $(\mathbf I - \mathbf A_\epsilon )^{-1}$ is bounded on $L^p$, and also $\mathbf \Sz_\epsilon$, it sufficies to see that $\mathbf S\mathbf R_\epsilon$ is also bounded on $L^p$. Assume for the moment that $p \leq 2$. Then since $\mathbf R_\epsilon$ maps $L^1$ to $L^\infty$, it also maps $L^p$ to $L^2$ (this follows from the inclusions of Lebesgue spaces, which hold in this setting because $D$ is bounded), while $\mathbf S$ maps $L^2$ to itself, yielding the fact that $\mathbf S\mathbf R_\epsilon$ is bounded on $L^p$. The case $2 \leq p$ is obtained by dualizing this argument.\\

The proof of theorem \ref{T:Bergman-bdd} can be found in \cite{LS-2}: it has an outline similar
to the proof of theorem \ref{T:Szego-bdd} with the operators $\mathbf \B_\epsilon$, see \eqref{D:HRL-op}, now in place of the $\mathbf\Sz_\epsilon$, but the details are simpler since we are dealing with operators that converge absolutely (as suggested by corollary \ref{C:abs-Bergman}). Thus one can avoid the delicate $T(1)$-theorem machinery and make instead absolutely convergent integral estimates. 
\bigskip

\section{the Cauchy-Leray integral revisited}\label{S:last}

 For domains with
  boundary regularity
    below the $C^2$ category there is no canonical
    notion of strong pseudo-convexity - much less a working
analog of the Cauchy-type operators $\mathbf\Sz_\epsilon$
and $\mathbf\B_\epsilon$ that were introduced in the previous sections. By contrast, the Cauchy-Leray integral can be defined for 
less regular domains, but the 
definitions and the proofs are substantially more delicate than the
 $C^2$ framework of theorem \ref{T:CL-bdd-smooth}.
 \begin{Def}\label{D:domain-C-1-1}
 Given a bounded domain 
$D\subset \mathbb C^n$, we say that {\em $D$ is of class $C^{1, 1}$} if $D$ has a defining
function (in the sense of definition \ref{D:domain-C-k}) that is of class $C^{1,1}$ in a neighborhood  $U$ of $\bndry D$; that is, $\rho$ is of class $C^1$ and its (real) partial derivatives 
$\dee\rho/\dee x_j$ 
are Lipschitz functions with respect to the Euclidean distance in $\mathbb C^n \equiv \mathbb R^{2n}$:
$$
\left|\frac{\dee\rho}{\dee x_j}(w) - \frac{\dee\rho}{\dee x_j}(\zeta)\right| \leq C|w-\zeta|\quad
w, \zeta\in U,\ \ j=1,\ldots, 2n.
$$
\end{Def}\smallskip

\begin{Thm}\label{T:CL-bdd-C11}
   Suppose $D$ is a bounded domain of class $C^{1,1}$ which is strongly $\mathbb{C}$-linearly convex. Then there is a natural definition of the Cauchy-Leray integral \eqref{D:CL-smooth}, 
     so that the mapping $f \mapsto \mathbf \Sz_L(f)$ initially defined for $f \in C^1 (\bndry D )$, extends to a bounded operator on $L^p (\bndry D, d\sigma)$ for $1 < p < \infty$. 
\end{Thm}

Note that in comparison with theorems \ref{T:Szego-bdd} and \ref{T:Bergman-bdd}, here our hypotheses about the nature of convexity are stronger, but the regularity of the boundary is weaker.\\

First, we explain the main difficulty in defining the Cauchy-Leray integral in the case of $C^{1,1}$ domains. It arises from the fact that the definitions \eqref{E:szego} and  \eqref{D:levi-leray} involve {\em second} derivatives of the defining function $\rho$. However $\rho$ is only assumed to be of class $C^{1,1}$, so that these derivatives are $L^{\infty}$ functions on $\mathbb{C}^n$, and as such not defined on $\bndry D$ which has 2$n$-dimensional Lebesgue measure zero.
What gets us out of this quandary is that here in effect not all second derivatives are involved but only those that are \textquotedblleft tangential\textquotedblright \ to $\bndry D$. Matters are made precise by the following \textquotedblleft restriction\textquotedblright \ principle and its variants. \\

Suppose $F \in C^{1,1}(\mathbb{C}^n )$ and we want to define $\displaystyle \overline{\partial} \partial F \bigg|_{\bndry D}$. We note that if $F$ were of class $C^2$ we would have 
\begin{equation}\label{Eq:7}
       \int_{\bndry D}\limits j^\ast (\overline{\partial} \partial F) \wedge \Psi = - \int_{\bndry D}\limits j^\ast (\partial F) \wedge d \Psi ,                                                                                                                                                \end{equation} 
where $\Psi$ is any $2n - 3$ form of class $C^1$, and here $j^\ast$ is the induced mapping to forms on $\bndry D$. 
\begin{Prop}
   For $F \in C^{1,1} (\mathbb{C}^n )$, there exists a unique 2-form $j^\ast (\overline{\partial} \partial F)$ in $\bndry D$ with $L^\infty (d \sigma )$ coefficients so that \eqref{Eq:7} holds.\\
\end{Prop}
\indent This is a consequence of an approximation lemma: There is a sequence $\{ F_n \}$ of $C^\infty$ functions on $\mathbb{C}^n$, that are uniformly bounded in the $C^{1,1} (\mathbb{C}^n )$ norm, so that $F_k \rightarrow F$ and $\triangledown F_k \rightarrow \triangledown F$ uniformly on $\bndry D$, and moreover $\triangledown_T^2 F_n$ converges $(d\sigma )$ a.e. on $\bndry D$. Here $\triangledown_T^2 F$ is the \textquotedblleft tangential\textquotedblright \ restriction of the Hessian $\triangledown^2 F$ of $F$. Moreover the indicated limit, which we may designate as $\triangledown_T^2 F$, is independent of the approximating sequence $\{ F_n \}$.\\

\end{document}